\documentclass[11pt]{amsart} %

\usepackage{amsmath,amsxtra,amssymb,graphicx,array,amsthm,pifont,units,framed,mathtools,tensor,mathrsfs,bm}

\usepackage[shortlabels]{enumitem}
\usepackage{mathpazo}
\usepackage[normalem,normalbf]{ulem}
\usepackage[b]{esvect}
\usepackage[all]{xy}

\usepackage[british]{babel}\usepackage[utf8]{inputenc}\usepackage[T1]{fontenc}

\usepackage[table]{xcolor}

\usepackage[hmargin={2.5cm},vmargin={1.3cm,1.5cm}]{geometry}

\usepackage{etoolbox}

\patchcmd{\subsection}{\bfseries}{\bfseries}{}{}
\patchcmd{\subsection}{-.5em}{.5em}{}{}

\theoremstyle{theorem}
\newtheorem{thm}{Theorem}[section]

\newtheorem{prop}[thm]{Proposition}

\newtheorem{lem}[thm]{Lemma}

\newtheorem{pty}[thm]{Property}

\newtheorem{csq}[thm]{Consequence}

\newtheorem{corollary}[thm]{Corollary}

\theoremstyle{definition}
\newtheorem{dfn}[thm]{Definition}

\newtheorem{cond}[thm]{Condition}

\theoremstyle{remark}

\newtheorem{remark}[thm]{Remark}

\newtheorem{example}{Example}

\usepackage[pdfpagelabels=true,ocgcolorlinks,colorlinks=true,linkcolor=blue,citecolor=purple,linktocpage,bookmarks=true,hyperfootnotes=false,plainpages=false]{hyperref}

\newcommand{\mg}{\operatorname{\mathfrak{d}}\nolimits}
\newcommand{\spn}{\operatorname{span}\nolimits}
\newcommand{\nn}{\ensuremath{\mathbb{N}}}
\newcommand{\B}{\ensuremath{\mathcal{B}}}
\newcommand{\R}{\ensuremath{\mathcal{R}}}
\newcommand{\ppq}{\leqslant}
\newcommand{\pgq}{\geqslant}
\newcommand{\ov}{\overline}
\newcommand{\car}{\operatorname{char}\nolimits}
\newcommand{\mt}{\ensuremath{\mathfrak{t}}}
\newcommand{\mo}{\ensuremath{\mathfrak{o}}}
\newcommand{\rrad}{\mathfrak{r}}
\newcommand{\op}{\text{op}}
\newcommand{\ot}{\ov\otimes}
\newcommand{\ots}{\ov\otimes\cdots\ov\otimes}
\newcommand{\gldim}{\operatorname{gldim}\nolimits}
\renewcommand{\ker}{\operatorname{Ker}\nolimits}
\newcommand{\im}{\operatorname{Im}\nolimits}
\newcommand{\Ext}{\operatorname{Ext}\nolimits}
\newcommand{\HH}{\operatorname{HH}\nolimits}
\newcommand{\Hom}{\operatorname{Hom}\nolimits}
\newcommand{\ld}[1]{\tensor*[^{*}]{#1}{}}
\newcommand{\lo}[1]{\tensor*[^{\perp}]{#1}{}}
\newcommand{\lkd}[1]{\tensor*[^{\natural}]{#1}{}}
\newcommand{\ulkd}[1]{\tensor*[^{!}]{#1}{}}

\newcommand{\ssq}{\ensuremath{\Lambda _0}} 

\newcommand{\set}[1]{\left\{ #1 \right\}}
\newcommand{\rep}[1]{\langle{#1}\rangle} 
\newcommand{\abs}[1]{\left|{#1}\right|} 

\newcommand{\df}[1]{{\itshape#1}}

\newcommand{\prj}{\ensuremath{\mathbf{P}}}
\newcommand{\diff}{b}
\newcommand{\rln}{\ensuremath{\mathcal{S}}}
\newcommand{\tns}{\ensuremath{\mathbb{T}}}

\usepackage{xspace}
\newcommand{\aloop}{$A$-loop\xspace}
\newcommand{\acycle}{$A$-cycle\xspace}
\newcommand{\apath}{$A$-path\xspace}
\newcommand{\atrail}{$A$-trail\xspace}
\newcommand{\alength}{$A$-length\xspace}
\newcommand{\asubpath}{$A$-subpath\xspace}
\newcommand{\asubcycle}{$A$-subcycle\xspace}

\newcommand{\aloops}{$A$-loops\xspace}

\newcommand{\apaths}{$A$-paths\xspace}
\newcommand{\atrails}{$A$-trails\xspace}

\newcommand{\asubpaths}{$A$-subpaths\xspace}

\newcommand{\cQ}{{\mathcal Q}}

\newcommand{\addappendix}[1]{%
  \section*{\appendixname.\  #1}
  \addcontentsline{toc}{section}{\appendixname}
  \stepcounter{section}
  \renewcommand{\thesection}{A}
}

\allowdisplaybreaks[2]

\begin{document}

\title[A combinatorial characterisation of (Fg)]
{A combinatorial characterisation of $d$-Koszul and $(D,A)$-stacked monomial algebras that satisfy \textbf{(Fg)}}

\author[Jawad]{Ruaa Jawad}
\address{Ruaa Jawad\\
Technical Instructors Training Institute, Middle Technical University \\
Baghdad, Iraq}
\email{ruaa.yousuf.jawad@mtu.edu.iq}
\author[Snashall]{Nicole Snashall}
\address{Nicole Snashall\\
School of Computing and Mathematical Sciences\\
University of Leicester \\
University Road \\
Leicester LE1 7RH \\
United Kingdom}
\email{njs5@leicester.ac.uk}
\author[Taillefer]{Rachel Taillefer}
\address{Rachel Taillefer\\ Universit\'e Clermont Auvergne, CNRS, LMBP, F-63000 Clermont-Ferrand\\
France}
\email{rachel.taillefer@uca.fr}

\thanks{Some of these results formed part of the first author's PhD thesis at the University of Leicester, which was supported by The Higher Committee For Education Development in Iraq (HCED) Reference D1201116.}

\makeatletter
\@namedef{subjclassname@2020}{%
  \textup{2020} Mathematics Subject Classification}
\makeatother


\begin{abstract}
Condition \textbf{(Fg)} was introduced in \cite{EHSST} to ensure that the theory
of support varieties of a finite dimensional algebra, established by Snashall and Solberg, has some similar properties to that of a group algebra. In this paper we give some easy to check combinatorial conditions that are equivalent to \textbf{(Fg)}  for monomial $d$-Koszul algebras. We  then extend this to monomial $(D,A)$-stacked algebras. We also extend the description of the Yoneda algebra of a $d$-Koszul algebra in \cite{GMMVZ} to $(D,A)$-stacked monomial algebras.
\end{abstract}

\date{\today}
\maketitle

\section*{Introduction}

Let $\Lambda $ be an indecomposable finite dimensional algebra over a field $K$. 
Support varieties for modules over  $\Lambda $ were introduced by  Snashall and Solberg in \cite{SS}, as a geometric tool to study the representation theory of $\Lambda $, using the  Hochschild cohomology $\HH^*(\Lambda )$. It was then proved in \cite{EHSST} that many of the properties of support varieties for group algebras have analogues in this more general case, provided some finiteness conditions hold. These are now known as \textbf{(Fg)} and can be expressed  in the following way. Let $\rrad$ be the Jacobson radical of $\Lambda $ and let $E(\Lambda )=\Ext_\Lambda ^*(\Lambda /\rrad,\Lambda /\rrad)$ be its Yoneda algebra. Then Condition \textbf{(Fg)} states that: 
\begin{center}
\textbf{(Fg)}   \begin{minipage}{0.8\linewidth}
    \begin{center}
there is a commutative Noetherian graded subalgebra $H$ of $\HH^*(\Lambda )$ with $H^0=\HH^0(\Lambda )$ such that
 $E(\Lambda )$ is a finitely generated $H$-module.
\end{center}
\end{minipage}
\end{center}
In particular, it was shown in \cite{EHSST} that if $\Lambda $
satisfies Condition \textbf{(Fg)}, then $\Lambda $ is necessarily Gorenstein,
that the variety of a module is trivial if and only if the module has finite
projective dimension, and periodic modules can be characterised up to projective
summands as those whose support variety is a line. Moreover, the converse of the first result mentioned was proved for monomial algebras in \cite{dgt}, that is, a Gorenstein monomial algebra satisfies \textbf{(Fg)}.

Support varieties for group algebras have been  very effective in the study of the representations of these algebras. Therefore Condition \textbf{(Fg)} has been much studied as it ensures a similarly useful theory of support varieties for finite dimensional algebras. For instance,  Condition \textbf{(Fg)} is invariant under various constructions, such as derived equivalence or singular equivalence of Morita type, see \cite{KPS,Sk,PSS}. Condition \textbf{(Fg)} has been studied or shown to hold for large families of algebras in \cite{ST1,W,ST2,Erdmann} among others, and support varieties have been studied for algebras that satisfy \textbf{(Fg)}, see for instance \cite{FS,SchrollS}. 

Since Hochschild cohomology is generally very difficult to compute, Condition \textbf{(Fg)}  can be difficult to establish for a given algebra. It is therefore useful to have necessary, sufficient or equivalent conditions for \textbf{(Fg)} to hold for a given algebra. One such result was proved by Erdmann and Solberg in \cite{ES}, where they showed that if \textbf{(Fg)} holds for $\Lambda $, then the graded centre $Z_{\text{gr}}(E(\Lambda ))$ of the Yoneda algebra is a Noetherian algebra and $E(\Lambda )$ is a finitely generated $Z_{\text{gr}}(E(\Lambda ))$-module; moreover, they proved that this is an equivalence when the algebra $\Lambda $ is Koszul. 
For monomial algebras, \textbf{(Fg)} was proved in \cite{dgt} to be equivalent to the related condition that the  $A_\infty$-centre $Z_\infty(E(\Lambda ))$ is  a Noetherian algebra and $E(\Lambda )$ is a finitely generated $Z_{\infty}(E(\Lambda ))$-module. 
We note that if $\Lambda$ has finite global dimension then both $E(\Lambda )$ and $\HH^*(\Lambda )$ are finite dimensional as vector spaces, and $\Lambda$ has \textbf{(Fg)}. Thus we are particularly interested in algebras of infinite global dimension.

The aim of this paper is to prove that a number of conditions are equivalent to
\textbf{(Fg)} for a large category of algebras, namely finite dimensional
$d$-Koszul, and more generally $(D,A)$-stacked, monomial algebras. This is
motivated in particular by a result of the first author in her PhD
thesis \cite{J}, a result that provides a sufficient and not difficult to check condition for $d$-Koszul monomial algebras to satisfy \textbf{(Fg)}; this result is Theorem \ref{tm:J-sufficient} in this paper. 

Berger introduced $d$-Koszul algebras in \cite{berger} as a natural generalisation of Koszul algebras (which occur as $2$-Koszul algebras). They are the algebras such that the $n$-th projective module in a  minimal projective resolution of $\Lambda /\rrad$
 as a $\Lambda $-module is generated in a specific degree denoted by $\delta (n)$ (with $\delta (n)=n$ if $d=2$). Moreover, they were characterised in \cite{GMMVZ} as the algebras $\Lambda $ that are $d$-homogeneous (that is, their ideal of relations can be generated by a set of homogeneous elements of degree $d$) and such that $E(\Lambda )$ is generated in degrees $0$, $1$ and $2$. The $(D,A)$-stacked monomial algebras, where $D\pgq 2$ and $A\pgq 1$ are integers, were introduced by Green and Snashall in \cite{GS-colloq math}, and those of infinite global dimension were characterised by the same authors in \cite{GS-J Alg} as the monomial algebras  such that the $n$-th projective module in a minimal projective resolution of $\Lambda /\rrad$
 as a $\Lambda $-module is generated in precisely one degree and such that $E(\Lambda )$ is finitely generated (in which case $E(\Lambda )$ is generated in degrees $0$, $1$, $2$ and $3$). In particular, when $A=1$, a $(D,1)$-stacked monomial algebra is $D$-Koszul. Thus $(D,A)$-stacked monomial algebras are natural generalisations of  $d$-Koszul and indeed Koszul monomial algebras.

In this paper, we consider Condition \textbf{(Fg)} for $d$-Koszul monomial algebras and more generally for $(D,A)$-stacked monomial algebras. We introduce some combinatorial conditions \ref{cond:comb-d-koszul} (in the $d$-Koszul case) and \ref{cond:comb-da} (in the $(D,A)$-stacked case) that are easy to check in terms of a minimal set of relations for the algebra $\Lambda $, and we prove that they are equivalent to \textbf{(Fg)} when $K$ is algebraically closed. This gives a very practical way of checking whether a monomial $d$-Koszul or $(D,A)$-stacked algebra satisfies \textbf{(Fg)}, because it is easy to check that a monomial algebra is $d$-Koszul using \cite[Theorem 10.2]{GMMVZ} (recalled in Property \ref{pty:d-covering}) or $(D,A)$-stacked using \cite[Section 3]{GS-colloq math} (recalled in Property \ref{pty:da-stacked-covering}).

To summarise, if  $\Lambda $ is a finite dimensional monomial algebra over an algebraically closed field $K$, which is $d$-Koszul with $d\pgq 2$ or $(D,A)$-stacked with $D\neq 2A$ whenever  $A > 1$, then the following conditions are equivalent:
\begin{enumerate}[\bfseries(C1)]
\item\label{cond:fg} $\Lambda $ satisfies \textbf{(Fg)}.
\item\label{cond:J-da-stacked} $\Lambda $ satisfies some combinatorial
  conditions defined in Condition \ref{cond:comb-d-koszul} when
  $\Lambda $ is $d$-Koszul monomial and in Condition \ref{cond:comb-da} when $\Lambda $ is  $(D,A)$-stacked monomial  (this follows from Theorems \ref{tm:characterisations-fg-d-koszul} and \ref{tm:characterisations-fg-da-stacked}).
\item\label{cond:Zgr} $Z_{\text{gr}}(E(\Lambda ))$ is Noetherian and $E(\Lambda
  )$ is a finitely generated  $Z_{\text{gr}}(E(\Lambda ))$-module (by \cite{ES} and Theorems \ref{tm:characterisations-fg-d-koszul} and \ref{tm:characterisations-fg-da-stacked}).
\item\label{cond:Zgr-light} $E(\Lambda
  )$ is a finitely generated  $Z_{\text{gr}}(E(\Lambda ))$-module (again by Theorems \ref{tm:characterisations-fg-d-koszul} and \ref{tm:characterisations-fg-da-stacked}).
\item\label{cond:Zinf}  $Z_{\infty}(E(\Lambda ))$ is Noetherian and $E(\Lambda )$ is a finitely generated  $Z_{\infty}(E(\Lambda ))$-module (by \cite{dgt}).
\item \label{cond:gorenstein}$\Lambda $ is Gorenstein (by \cite{dgt}).
\end{enumerate}
The assumption that $K$ is algebraically closed is needed for the implication \ref{cond:fg}$\Rightarrow $\ref{cond:Zgr} of \cite{ES} which we use in our proof of \ref{cond:fg}$\Rightarrow $\ref{cond:J-da-stacked}; however, \ref{cond:J-da-stacked} implies \ref{cond:fg} without this assumption.

The paper is organised as follows. In Section \ref{section:background} we give
some background on monomial algebras and the notion of overlaps, as well as on
the Yoneda algebra and the Hochschild cohomology of a monomial algebra. Section
\ref{section:d-koszul} is devoted to the proof of the implications
\ref{cond:J-da-stacked}$\Rightarrow $\ref{cond:fg} and
\ref{cond:Zgr-light}$\Rightarrow $\ref{cond:J-da-stacked} for $d$-Koszul
monomial algebras, which completes the equivalence of all the conditions
above. The first implication relies on a presentation of the Hochschild
cohomology for $(D,A)$-stacked monomial algebras from \cite{GS-colloq math} and
the second one uses a description of the Yoneda algebra $E(\Lambda )$ of a
$d$-Koszul algebra $\Lambda $ as a graded subspace  of the Koszul dual algebra
$\ulkd\Lambda $ from \cite{GMMVZ}. In Section \ref{sec:da-stacked}, we extend
these results to $(D,A)$-stacked monomial algebras  where $D \neq 2A$ whenever $A > 1$. Here again, we use a description of $E(\Lambda )$ as a subspace of an analogue $\lkd \Lambda $ of the Koszul dual of $\Lambda $; this description is detailed and proved in the appendix, and is a generalisation of the corresponding result of \cite{GMMVZ} to $(D,A)$-stacked monomial algebras. 

\subsubsection*{General assumptions.} Throughout the paper, $\Lambda$ is an indecomposable finite dimensional algebra over a field $K$ with $\car (K) \neq 2$ that is not necessarily algebraically closed.  Moreover, we assume that $\Lambda = K\cQ/I$ where $\cQ$ is a finite quiver (it has a finite number of vertices and arrows) and $I$ is an admissible ideal in $K\cQ$. If $\Lambda = K\cQ/I$ is also a monomial algebra then $I$ is generated by a minimal set $\rho$ of paths (monomials) and $\Lambda$ is graded by the length of paths; we denote by $\ell(p)$ the length of a path $p$. Note that paths in any algebra given by quiver and relations are written from left to right. 
For any $j\pgq 0$, we shall denote by $\cQ_j$ the set of paths of length $j$ in $\cQ$.

In order to use the results in \cite{GS-colloq math}, we shall need to assume
that $\gldim \Lambda \pgq 4$. However, if $\Lambda $ is a monomial algebra with
finite global dimension, all the conditions \ref{cond:fg}--\ref{cond:gorenstein}
hold for $\Lambda $ (we note that Condition
\ref{cond:J-da-stacked} is necessarily empty in this case). Therefore we do not lose any generality in making this assumption.

\section{Some background on monomial algebras and their cohomology}\label{section:background}

\subsection{Overlaps}\label{subsec:overlaps}
Keeping the above assumptions, let $\Lambda = K\cQ/I$ be a monomial algebra so that $\Lambda = \oplus_{i\pgq 0}\Lambda_i$ is a graded algebra with the length grading. We denote by $\rrad=\oplus_{i\pgq 1}\Lambda_i$ the radical of $\Lambda $.
An arrow $\alpha$ starts at the vertex $\mathfrak{o}(\alpha)$ and ends at
the vertex $\mathfrak{t}(\alpha)$.
If $p = \alpha_1\alpha_2\cdots \alpha_n$ is a path with
$\alpha_1, \alpha_2, \ldots , \alpha_n$ in $\cQ_1$ then
$\mo(p) = \mathfrak{o}(\alpha_1)$ and $\mt(p) = \mathfrak{t}(\alpha_n)$.

A path $p$ is a {\it prefix} of a path $q$ if there is some path $p'$ such that $q = pp'$; if an arrow $\alpha$ is a prefix of $q$ then we say that $q$ {\it begins} with $\alpha$.
A path $p$ is a {\it suffix} of a path $q$ if there is some path $p'$ such that $q = p'p$; if an arrow $\alpha$ is a suffix of $q$ then we say that $q$ {\it ends} with $\alpha$.

\medskip

We use the concept of overlaps of \cite{GHZ} and \cite{GZ} to describe the minimal projective resolution of $\Lambda _0\cong\Lambda/\rrad$ over $\Lambda$, and to describe the minimal projective resolution of $\Lambda$ over $\Lambda^e$, where $\Lambda^e$ is the enveloping algebra $\Lambda^{\op}\otimes_K\Lambda$ of $\Lambda$. We recall the relevant definitions here using the notation of \cite{GS-colloq math}.

\begin{dfn}\label{defin:overlaps}
\begin{enumerate}
\item A path $q$ overlaps a path $p$ with overlap $pu$ if there are paths
$u$ and $v$ such that $pu = vq$ and $1 \leqslant \ell(u) < \ell(q)$. We illustrate the
definition with the following diagram.
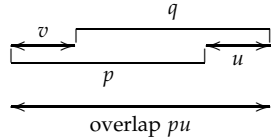
\[\xymatrix@W=0pt@M=0.3pt{
\ar@{^{|}-^{|}}@<-1.25ex>[rrr]_p\ar@{{<}-{>}}[r]^{v} \ar@{{<}-{>}}@<-4.5ex>[rrrr]_{\text{overlap }pu} &
\ar@{_{|}-_{|}}@<1.25ex>[rrr]^q & & \ar@{{<}-{>}}[r]_{u} &
& }\]
Note that we allow $\ell(v) = 0$ here.
\item A path $q$ properly overlaps a path $p$ with overlap $pu$ if $q$ overlaps $p$ and $\ell(v) \geqslant 1$.
\item A path $p$ has no overlaps with a path $q$ if $p$ does not properly overlap $q$
and $q$ does not properly overlap $p$.
\end{enumerate}
\end{dfn}

We now define sets $\R^n$ recursively. Let
\[\begin{array}{lll}
\R^0 & = & \mbox{$\cQ_0$, the set of vertices of ${\mathcal Q}$}\\
\R^1 & = & \mbox{$\cQ_1$, the set of arrows of ${\mathcal Q}$}\\
\R^2 & = & \mbox{$\rho$, the minimal generating set for $I$.}
\end{array}\]
For $n\pgq 3$, the construction is as follows.

\begin{dfn}
\begin{enumerate}
\item For $n \geqslant 3$, we say that $R^2 \in \R^2$ maximally overlaps $R^{n-1} \in \R^{n-1}$
with overlap $R^n = R^{n-1}u$ if
\begin{enumerate}
\item $R^{n-1} = R^{n-2}p$ for some path $p$;
\item $R^2$ overlaps $p$ with overlap $pu$;
\item there is no element of $\R^2$ which overlaps $p$ with overlap being a
proper prefix of $pu$.
\end{enumerate}
We may also say that $R^n$ is a maximal overlap of $R^2 \in \R^2$ with
$R^{n-1} \in \R^{n-1}$.\\
The construction of $R^n$ is illustrated in the following diagram.
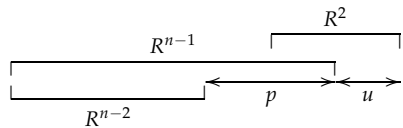
\[\xymatrix@W=0pt@M=0.3pt{
\ar@{^{|}-^{|}}@<-1.25ex>[rrr]_{R^{n-2}}\ar@{_{|}-_{|}}@<1.5ex>[rrrrr]^{R^{n-1}} & & &
\ar@{{<}-{>}}[rr]_{p} &
\ar@{_{|}-_{|}}@<3.5ex>[rr]^{R^2} &
\ar@{{<}-{>}}[r]_{u} &
& }\]

\item For $n \geqslant 3$, the set $\R^n$ is defined to be the set of all overlaps $R^n$ formed in this way.
\end{enumerate}
\end{dfn}

We also recall from \cite{GZ} that if $R_1^np = R_2^nq$, for $R_1^n, R_2^n \in \R^n$ and
paths $p, q$, then $R_1^n = R_2^n$ and $p = q$. Any element $R^n$ in $\R^n$ may be expressed
uniquely as $R_j^{n-1}a_j$ and as $b_kR_k^{n-1}$ for some $R_j^{n-1}, R_k^{n-1}
$ in $\R^{n-1}$ and paths $a_j, b_k$. We say that the elements $R_j^{n-1}$ and
$R_k^{n-1}$ {\it occur in} $R^n$.

\subsection{The Ext algebra $E(\Lambda)$}\label{subsec:Ext algebra}
The Ext algebra $E(\Lambda)$ is given by $E(\Lambda) = \Ext^*_{\Lambda}(\ssq, \ssq) = \bigoplus_{n\pgq 0}\Ext^n_{\Lambda}(\ssq, \ssq)$
with the Yoneda product. In the terminology of overlaps, the $n$-th projective module in a minimal projective $\Lambda$-resolution of $\ssq$ is $\bigoplus_{R^n \in\R^n}\mt(R^n)\Lambda$. Then $\Ext^n_{\Lambda}(\ssq, \ssq)$ has a basis indexed by $\R^n$ and $E(\Lambda)$ has a basis indexed by $\bigcup_{n\pgq 0}\R^n$ (see \cite{GZ, GHZ}). We identify $R^n_i \in\R^n$ with the corresponding element of $\Ext^n_{\Lambda}(\ssq, \ssq)$, that is, with the map
$\bigoplus_{R^n \in\R^n}\mt(R^n)\Lambda \to \ssq$ given by
\[\mt(R^n)\lambda \mapsto
\begin{cases}
\mt(R^n_i)\lambda +\rrad & \mbox{if $R^n = R^n_i$}\\
0 & \mbox{otherwise.}
\end{cases}\]

\subsection{The Hochschild cohomology ring $\HH^*(\Lambda)$}\label{subsec:HH}

Let $({\mathcal P}^*, \partial ^*)$ be the minimal projective $\Lambda^e$-resolution of $\Lambda$ from
\cite{B}. We write $\otimes$ for $\otimes_K$ throughout. Then
\[{\mathcal P}^n = \bigoplus_{R^n \in\R^n}\Lambda\mo(R^n)\otimes\mt(R^n)\Lambda.\]
The maps are given as follows.
In odd degrees, if $R^{2n+1} = R_j^{2n}a_j = b_kR_k^{2n} \in \R^{2n+1}$ then
$\partial ^{2n+1}\colon {\mathcal P}^{2n+1} \rightarrow {\mathcal P}^{2n}$
is given by
\[\mo(R^{2n+1})\otimes \mt(R^{2n+1}) \mapsto
\mo(R_j^{2n})\otimes a_j - b_k\otimes \mt(R_k^{2n})\]
where the first tensor lies in the summand
corresponding to $R_j^{2n}$ and the second tensor
lies in the summand corresponding to $R_k^{2n}$.

For even degrees, any element $R^{2n}$ in $\R^{2n}$ may be expressed in
the form $p_jR_j^{2n-1}q_j$ for some $R_j^{2n-1} \in \R^{2n-1}$ and paths $p_j, q_j$
with $n \pgq 1$. Let $R^{2n} = p_1R_1^{2n-1}q_1 = \cdots
= p_rR_r^{2n-1}q_r$ be all expressions of $R^{2n}$ which contain some
element of $\R^{2n-1}$ as a subpath.
Then, for $R^{2n} \in \R^{2n}$, the map
$\partial ^{2n}\colon {\mathcal P}^{2n} \rightarrow {\mathcal P}^{2n-1}$
is given by
\[\mo(R^{2n})\otimes \mt(R^{2n}) \mapsto
\sum_{j=1}^r p_j \otimes q_j\]
where the tensor $p_j \otimes q_j$ lies in the summand of ${\mathcal P}^{2n-1}$
corresponding to $R_j^{2n-1}$.

If not specified, then it will always be clear from the context in which summand of a
projective module our tensors lie.

The Hochschild cohomology ring $\HH^*(\Lambda)$ of $\Lambda$ is given by
\[\HH^*(\Lambda) = \Ext^*_{\Lambda^e}(\Lambda, \Lambda) = \bigoplus_{n\pgq 0}\Ext^n_{\Lambda^e}(\Lambda, \Lambda)\]
with the Yoneda product.

\section{Characterisations of $d$-Koszul monomial algebras that satisfy \textbf{(Fg)}}\label{section:d-koszul}

\subsection{Notation and properties of $d$-Koszul monomial algebras}

Let  $\Lambda=K\cQ /I $ be a monomial
algebra, where $\cQ $ is a finite quiver and $I$ is an admissible ideal in
$K\cQ $ generated by a minimal set $\rho $ of paths. Recall that the algebra $\Lambda =\bigoplus_{i\pgq 0}\Lambda _i$ is graded by the length of paths.  We can express $\Lambda $
 as a quotient $\Lambda =\tns_{\Lambda _0}(\Lambda _1)/I$ of the tensor algebra,
 where $\Lambda /\rrad\cong\Lambda _0=K\cQ _0$ and $\Lambda _1=K\cQ _1$ and  $I$ is an ideal generated by
 a minimal set $\rho$ of monomials.  The algebra $\Lambda _0\cong K^{\abs{\cQ _0}}$ is isomorphic to a finite product of copies of the base field $K$; it is therefore a semisimple and commutative $K$-algebra. We denote by $e_i$ the idempotent in $\Lambda _0$ corresponding to the vertex $i$.

Let $d\pgq 2$ be an integer. We assume that $\Lambda $ is a $d$-Koszul algebra, that is,  for any minimal projective right $\Lambda $-module resolution of $\Lambda _0$, the $n$-th projective module is generated in degree $\delta (n)$ where 
\[ \delta (n)=
\begin{cases}
\frac{n}{2}d& \text{if $n$ is even}\\
\frac{n-1}{2}d+1&\text{if $n$ is odd}.
\end{cases}
 \]
It follows that $\Lambda $ is $d$-homogeneous  (that is, $\rho $ consists of paths of length $d$).

The monomial $d$-Koszul  algebras can be characterised as follows. 

\begin{pty}\label{pty:d-covering}\cite[Theorem 10.2]{GMMVZ}
A finite dimensional $d$-homogeneous monomial algebra  $\Lambda=K{\cQ}/I$ is $d$-Koszul if, and only if, $\rho $ is \df{$d$-covering}, that
  is, for any paths $p$, $q$ and $r$ in $\cQ$,
  \[ (pq\in\rho ,\ qr\in\rho ,\ \ell(q)\pgq 1)\Rightarrow (\text{all subpaths of
    $pqr$ of length $d$ are in $\rho $}). \]
\end{pty}

Note that this condition is always satisfied if $d=2$; it is indeed well known that all finite dimensional quadratic monomial algebras are Koszul, see \cite{GZ} and \cite[Corollary 2.4.3]{pp}.

\begin{example}
\label{example:dKoszul}
  Let $\Lambda = K\cQ/I$ where $\cQ$ is the quiver
\[\xymatrix{
\cdot \ar[dr]^{\gamma_3} & & \\
\cdot \ar[u]^{\gamma_2} & \cdot \ar[r]^{\beta} \ar[l]^{\gamma_1} & \cdot\ar@(ur,dr)[]^{\alpha} 
}\]
and the ideal $I$ has minimal generating set $\rho = \{\alpha^3, \gamma_1\gamma_2\gamma_3, \gamma_2\gamma_3\gamma_1, \gamma_3\gamma_1\gamma_2\}$.
Then $\Lambda$ is a $3$-Koszul monomial algebra.
\end{example}

\smallskip

From now on, we assume that $\Lambda=K{\cQ}/I$ is a finite dimensional $d$-Koszul
monomial algebra with $d\pgq 2$.

We have the following consequences of Property \ref{pty:d-covering}.

\begin{csq}\label{csq:d-covering}\cite[Proposition 7.13]{J}
  Let $R_i^n$ be an element in
  $\R^n$. Then all subpaths of $R_i^n$ of length $d$ are in
  $\rho $. 
\end{csq}

\begin{proof}  
The result is proved by induction.  It is clear when $n=2$. Moreover, if $n=3$,
since $R_i^3\in\R^3$ is a maximal overlap of two elements in $\R^2$, it follows
from Property \ref{pty:d-covering}.

Now let $n\pgq 4$ be an integer and take $R_i^n\in \R^n$. Then $R_i^n$ is a
maximal overlap of $R_1^2\in\R^2$ with $R_2^{n-1}\in \R^{n-1}$ so that
$R_i^n=R_2^{n-1}u$ for some path $u$, and $R_2^{n-1}$ is a maximal overlap of
$R_3^2\in\R^2$ with $R_4^{n-2}\in\R^{n-2}$ so that $R_2^{n-1}=R_4^{n-2}u'$ for
some path $u'$. This can be illustrated as follows:
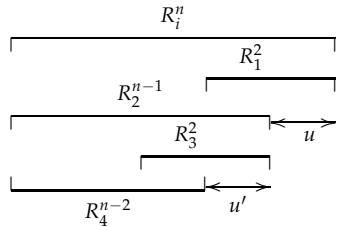
\[\xymatrix@W=0pt@M=0.3pt{\ar@{_{|}-_{|}}@<10ex>[rrrrr]^{R_i^n}
\ar@{^{|}-^{|}}@<-1.25ex>[rrr]_{R_4^{n-2}}
\ar@{_{|}-_{|}}@<4.25ex>[rrrr]^{R_2^{n-1}}
&&\ar@{_{|}-_{|}}@<1.25ex>[rr]^(.35){R_3^2} & \ar@{_{|}-_{|}}@<7ex>[rr]^(.35){R_1^2} \ar@{{<}-{>}}@<-1ex>[r]_{u'} & \ar@{{<}-{>}}@<3.75ex>[r]_(.6){u} 
&}\]Moreover, $\ell(u'u)=\ell(R_i^n)-\ell(R_4^{n-2})=\delta (n)-\delta (n-2)=d$ so $u'u=R_1^2$. By induction, every subpath of $R_2^{n-1}$ of length $d$ is in
$\rho $. Any other subpath of length $d$ of $R_i^n$ is either $u'u=R_1^2\in\R^2$ or  a proper subpath of $R_3^2u$,
therefore it is in $\rho $ by Property \ref{pty:d-covering}. 
We have proved  the induction step.  
\end{proof}

A \df{trail} in $\cQ$ is a path $T=\alpha_1\cdots\alpha_n$ with $n \pgq 1$ such that the arrows $\alpha_i$ are all distinct.
We say that the trail is {\it closed} when $\mt(\alpha_n) = \mo(\alpha_1)$. A path $q$ is said to {\it lie on} the closed trail $T$ if $q$ is a subpath of $T^m$ for some $m \pgq 1$.
We say that two trails are \df{distinct} if neither lies on the other. 

We now have a second consequence of Property~\ref{pty:d-covering}.

\begin{csq}\label{csq:trail}\cite[Proposition 7.14]{J}
Suppose that $T = \alpha_1 \cdots \alpha_n$ is a closed trail in ${\cQ}$ and that $d \pgq n+1$.
Then all paths of length $d$ that lie on the closed trail $T$ are in $\rho$.
\end{csq}

\begin{proof}
Since $\Lambda$ is finite dimensional, there is a path $R_2 \in \rho$ that lies on $T$. Now, $\ell(R_2) = d$ and $d \pgq n+1$ so, without loss of generality, we may suppose that $R_2 =
(\alpha_1\alpha_2 \cdots \alpha_n)^m\alpha_1\alpha_2 \cdots \alpha_s$ for some $1 \ppq s \ppq n$ with $d = nm + s$ and $m\pgq 1$.
Let $p = (\alpha_1\alpha_2 \cdots \alpha_n)^m, q = \alpha_1\alpha_2 \cdots \alpha_s$ and $r = (\alpha_{s+1}\cdots\alpha_n\alpha_1 \cdots \alpha_s)^m$.
Then $pq = R_2 = qr$ and we can apply Property~\ref{pty:d-covering} so that all subpaths of $pqr$ of length $d$ are in $\rho$.
Now, any path of length $d$ that lies on the closed trail $T$ is a subpath of $pqr$ and hence is in $\rho$.
\end{proof}

We now introduce Condition~\ref{cond:comb-d-koszul}. Jawad showed in her PhD thesis \cite{J} that this condition is
sufficient for $\Lambda$ to satisfy \textbf{(Fg)}; we give a proof in Theorem~\ref{tm:J-sufficient} below.

\begin{cond}\label{cond:comb-d-koszul}\cite[Theorems 7.11 and 7.15]{J} We say that a $d$-Koszul monomial algebra $\Lambda $ satisfies Condition \ref{cond:comb-d-koszul}, or \ref{cond:J-da-stacked},  when the following properties (1) and (2) both hold:
\begin{enumerate}[(1)]
\item Let $\alpha$ be a loop in $\cQ_1$. Then $\alpha^d\in\rho $ but there is no
  path in $\rho$ of the form $\alpha^{d-1}\beta $ or $\beta \alpha^{d-1}$
  where $\beta$ is an arrow that is distinct from $\alpha$.
\item Let $T=\alpha _1\cdots\alpha_n$ be a closed trail in $\cQ$ with $n>1$ and
  $\alpha_i\in \cQ_1$ for all $i$ and such that $\rho _T:=\{\alpha_1\cdots\alpha _d,\alpha_2\cdots\alpha _d\alpha_{d+1},\ldots ,\alpha_n\alpha_1\cdots\alpha_{d-1}\}\subseteq\rho $. Then there are no elements in $\rho \setminus\rho_T$ which begin or end with the arrow $\alpha_i$, for all $i$.
\end{enumerate}
\end{cond}

\begin{remark} If $T=\alpha_1 \cdots \alpha_n$ is a closed trail then the subscript $i$ of $\alpha_i$ is taken modulo $n$ within the range $1 \ppq i \ppq n$. Thus $\rho_T$ is the set of all paths of length $d$ that lie on the closed trail $T$.
\end{remark}

\begin{remark} Suppose that Condition~\ref{cond:comb-d-koszul} is non-empty, that is, there is a loop or a closed trail with the given properties. Then the description of the projective modules in Section~\ref{subsec:Ext algebra} using overlaps shows that $\Lambda_0$ has infinite projective dimension as a $\Lambda$-module, and hence $\Lambda$ has infinite global dimension. 
\end{remark}

\subsection{Condition~\ref{cond:comb-d-koszul} is sufficient for $\Lambda $ to satisfy \textbf{(Fg)}}

The proof of Theorem~\ref{tm:J-sufficient} uses the description of the Hochschild cohomology ring modulo nilpotence of a $(D,A)$-stacked monomial algebra from \cite[Theorem 3.4]{GS-colloq math}. We recall the definition of a $(D,A)$-stacked monomial algebra in Subsection~\ref{subsec:da-stacked}.
The Hochschild cohomology ring modulo nilpotence is the quotient $\HH^*(\Lambda)/{\mathcal N}$ where ${\mathcal N}$ is the ideal of $\HH^*(\Lambda)$ that is generated by the homogeneous nilpotent elements. It is well-known that $\HH^*(\Lambda)$ is a graded commutative ring, so, since $\car(K) \neq 2$, every homogeneous element of odd degree squares to zero. Moreover, ${\mathcal N}$ is the set of all nilpotent elements of $\HH^*(\Lambda)$.
Our calculations involving $\HH^*(\Lambda)$ use the minimal projective $\Lambda^e$-resolution $({\mathcal P}^*, \partial ^*)$ of $\Lambda$ from \cite{B}; see Section~\ref{subsec:HH}.

Noting that a $d$-Koszul monomial algebra is a $(d,1)$-stacked monomial algebra (see \cite{GS-colloq math}),
we apply \cite[Theorem 3.4]{GS-colloq math} in the special case where $D=d$ and $A=1$, and this simplifies the hypotheses. Specifically, if there is a closed path $C$ in $\cQ$ with $C^{D/A} \in \rho$ then $C^d \in \rho$ and it is immediate that $C$ has length 1 and is necessarily a loop.

\begin{thm}\cite[Theorems 7.11 and 7.15]{J}\label{tm:J-sufficient}
Let $\Lambda = K\cQ/I$ be a finite dimensional $d$-Koszul
monomial algebra with $d\pgq 2$. Assume that $\Lambda$ satisfies
Condition~\ref{cond:comb-d-koszul}. Then $\Lambda$ satisfies \textbf{(Fg)}.
\end{thm}

\begin{proof} We keep the notation of Condition~\ref{cond:comb-d-koszul}.

Let $\alpha_1, \dots, \alpha_u$ be the loops in the quiver $\cQ$, and suppose that $\alpha_i$ is a loop at the vertex $v_i$. Since $\Lambda$ is a finite dimensional $d$-Koszul monomial algebra, $\alpha_i^d$ is necessarily in the minimal generating set $\rho$.
By Condition~\ref{cond:comb-d-koszul}~(1), for each $i=1, \dots, u$, there are no elements in $\rho$ of the form $\alpha_i^{d-1}\beta$ or $\beta\alpha_i^{d-1}$ where $\beta$ is an arrow that is distinct from $\alpha_i$.

We need to show that there are no overlaps of $\alpha_i^d$ with any element of $\rho\setminus\{\alpha_i^d\}$. This is immediate if $d=2$, so suppose that $d \pgq 3$.
If $R \in \rho\setminus\{\alpha_i^d\}$ and $R$ overlaps $\alpha_i^d$, then either $R = \alpha_i^sb$ or $R = b\alpha_i^s$ where
$1 \ppq s \ppq d-1$ and $b$ is a path of length $d-s$ which does not begin (respectively, end) with the arrow $\alpha_i$.
Suppose first that $R = \alpha_i^sb$. Then $R$ overlaps $\alpha_i^d$ with overlap of length $2d-s$ as follows:
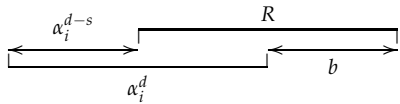
\[\xymatrix@W=0pt@M=0.3pt{
\ar@{{<}-{>}}[rr]^{\alpha_i^{d-s}} \ar@{^{|}-^{|}}@<-1.25ex>[rrrr]_{\alpha_i^d}& & \ar@{_{|}-_{|}}@<1.5ex>[rrrr]^{R} & & \ar@{{<}-{>}}[rr]_{b} & & }\]
This is a maximal overlap since $\alpha_i$ is not the first arrow of $b$ and thus gives an element $R^3_1 \in {\mathcal R}^3$. However, $\ell(R^3_1) = d+1$ since $\Lambda$ is $d$-Koszul. Thus $2d-s = d+1$ and so $s=d-1$. But then $R = \alpha_i^{d-1}b$ and $b$ is an arrow distinct from $\alpha_i$, which contradicts our hypothesis. The case where $R = b\alpha_i^s$ is similar. So there are no overlaps of $\alpha_i^d$ with any element of $\rho\setminus\{\alpha_i^d\}$. Moreover, as $\Lambda$ is a finite dimensional monomial algebra, it follows that the vertices $v_1, \dots, v_u$ are distinct.

\medskip

Let $T_{u+1}, \dots, T_r$ be the distinct closed trails in $\mathcal{Q}$ such that all paths of length $d$ that lie on these closed trails are contained in $\rho$. For each $i=u+1, \dots,r$,
we write $T_i=\alpha_{i,1}\cdots\alpha_{i,m_i}$, where the $\alpha_{i,j}$ are arrows, and set \[\rho_{T_i}=\{\alpha_{i,1}\cdots\alpha_{i,d}, \alpha_{i,2}\cdots\alpha_{i,d+1},\dots,  \alpha_{i,m_i}\alpha_{i,1}\cdots\alpha_{i,d-1}\}.\] Then $\rho_{T_i}$ is contained in $\rho$.
By Condition~\ref{cond:comb-d-koszul}~(2), for each closed trail $T_i$ ($i=u+1, \dots,r$), there are no elements in $\rho\setminus\rho_{T_i}$ which begin or end with the arrow $\alpha_{i,j}$, for all $j = 1, \dots, m_i$. So no arrow $\alpha_{i,j}$ has overlaps with any element in $\rho\setminus\rho_{T_i}$.

For $i = u+1, \dots , r$, let $T_{i,1},\dots,T_{i,m_i}$ be defined by 
\begin{align*}
&T_{i,1}=T_i=\alpha_{i,1}\alpha_{i,2}\cdots\alpha_{i,m_i}\\
&T_{i,2}= \alpha_{i,2}\alpha_{i,3}\cdots\alpha_{i,m_i}\alpha_{i,1}\\
&\qquad\vdots\\
&T_{i,m_i} =\alpha_{i,m_i}\alpha_{i,1} \cdots\alpha_{i,m_{i-1}}.
\end{align*}
Then the paths $T_{i,1},\dots,T_{i,m_i}$ are all of length $m_i$ and lie on the closed path $T_i$.

\medskip

\sloppy We now describe a  commutative Noetherian graded subalgebra $H$ of $\HH^*(\Lambda )$ with $H^0=\HH^0(\Lambda )$. 
As noted above, $\Lambda$ is a $(d,1)$-stacked monomial algebra. Moreover, Condition
\textbf{(Fg)} is always satisfied if  the global dimension of $\Lambda $ is
finite, therefore we may assume that $\gldim \Lambda \geqslant 4$. 
Hence we can apply \cite[Theorem 3.4]{GS-colloq math}, which gives
$\HH^*(\Lambda)/\mathcal{N}\cong K[x_1,\dots,x_r]/\langle x_ax_b \ \mbox{for} \ a\neq b\rangle$, where
\begin{itemize}
\item for $i=1,\dots,u$, the vertices $v_1,\dots,v_u$ are distinct and the element $x_i$ corresponding to the loop $\alpha_i$ is in degree $2$ and is represented by the map ${\mathcal P}^2\longrightarrow\Lambda$ where for $R^2\in\mathcal{R}^2$,
    \[\mo(R^2)\otimes\mt(R^2)\mapsto \begin{cases} v_i & \mbox{ if}\ R^2=\alpha_i^d \\
    0 & \mbox{ otherwise}
    \end{cases}\]
\item and for $i=u+1,\dots,r$, the element $x_i$ corresponding to the closed trail $T_i=\alpha_{i,1}\cdots\alpha_{i,m_i}$ is in degree $2\mu_i$ such that $\mu_i=m_i/\gcd(d,m_i)$ and is represented by the map ${\mathcal P}^{2\mu_i}\longrightarrow\Lambda$, where for $R^{2\mu_i}\in{\mathcal{R}{^{2\mu_i}}}$,
    \[\mo(R^{2\mu_i})\otimes\mt(R^{2\mu_i})\mapsto
\begin{cases}\mo(T_{i,k}) & \mbox{if}\ R^{2\mu_i}=T^{d/\gcd(d,m_i)}_{i,k}
\mbox{ for all } k=1,\dots,m_i  \\
0 &\mbox{otherwise}.
\end{cases}\]
\end{itemize}

Let $H$ be the subring of $\HH^*(\Lambda)$ generated by $Z(\Lambda)$ and $\{x_1,\dots,x_r\}$.
Since $Z(\Lambda) = \HH^0(\Lambda)$ and $\HH^*(\Lambda)$ is graded commutative, it follows that
\[H = Z(\Lambda)[x_1, \dots , x_r]/\langle x_ax_b \mbox{ for } a \neq b\rangle\] and so
$H$ is a commutative ring.
Moreover, $Z(\Lambda)$ is finite dimensional so is a commutative Noetherian ring. Thus $H$ is a Noetherian ring (see \cite[Corollary 8.11]{Sharp}). 

\medskip

The rest of this proof shows that $\Lambda$ satisfies  \rm\textbf{(Fg)}, with the algebra $H$ that we have just described.
Following the discussion in Section~\ref{subsec:Ext algebra}, we identify $\bigcup_{n\pgq 0}\R^n$ with a basis of $E(\Lambda)$. The action of a homogeneous element $x \in \HH^n(\Lambda)$ on $E(\Lambda)$ is then given by left multiplication by $\sum_jR^n_j$ where the sum is over all $j$ such that $x({\mathfrak o}(R^n_j) \otimes {\mathfrak t}(R^n_j)) \neq 0$.
Thus if $x_i \in \HH^2(\Lambda)$ corresponds to the loop $\alpha_i$, then the action of $x_i$ on $E(\Lambda)$ is given by left multiplication by $\alpha_i^d$. And if $x_i$ in degree $2\mu_i$ corresponds to the closed trail $T_i$, then the action of $x_i$ on $E(\Lambda)$ is given by left multiplication by $\sum_{k=1}^{m_i} T_{i,k}^{d/\gcd(d, m_i)}$.

Set $N=\max\{3, |x_1|,\dots,|x_r|, |\mathcal{Q}_1|\}$.
We show that $\bigcup_{n=0}^N\R^n$ is a generating
set for $E(\Lambda)$ as a left $H$-module and thus $E(\Lambda)$ is finitely generated as a left $H$-module.

Let $R\in \R^n$ with $n>N$. Then $\ell(R) = \delta(n) \pgq 2d$
and we can write $R=a_1a_2\cdots a_{\delta(n)}$ where the $a_i$ are in $\cQ_1$.
From Consequence~\ref{csq:d-covering}, all subpaths of $R$ of length $d$ are in $\rho$, so we may illustrate $R$ with the following diagram:
\[\xymatrix@W=0pt@M=0pt@C=0pt{
\ar@{^{|}-^{|}}@<-1.8ex>[rrrrr] & \, a_1\ar@{_{|}-_{|}}@<1.8ex>[rrrrr] & \, a_2 & \ \cdots \ & a_d\ar@{^{|}-^{}}@<-1.8ex>[rr]\, & \, a_{d+1} & \ \cdots \ & \ar@{^{|}-^{|}}@<-1.8ex>[rrrr] & \, a_{\delta(n)-d+1}&\ \cdots \ & a_{\delta(n)}\, & }\]
Now, $n>N\pgq\abs{\cQ_1}$ so there is some repeated arrow.
Choose $j,k$ with $k$ minimal and $k\pgq 1$ such that $a_{j}$ is a repeated arrow, $a_{j},\dots,a_{j+k-1}$ are all distinct arrows and $a_{j+k}=a_{j}$. Write
\[R=(a_1\cdots a_{j-1})(a_{j}\cdots a_{j+k-1})(a_{j}a_{j+k+1}\cdots a_{\delta(n)}).\]
There are two cases to consider.

{\bf Case (1):} $k=1$. Then $a_{j}=a_{j+1}$ and so $a_{j}$ is a loop. It follows that \[R=(a_1\cdots a_{j-1})(a_{j}a_{j})(a_{j+2}\cdots a_{\delta(n)}).\]

Suppose first that $j \ppq d-1$. Then $j+d-1 \ppq \delta(n)$ so from Consequence~\ref{csq:d-covering}, $a_{j}^2a_{j+2}\cdots a_{j+d-1}$ is in $\rho$.
But $a_j^d \in \rho$ and we have already shown that there are no overlaps of $a_j^d$ with any element of $\rho\setminus \{a_j^d\}$.
Thus $a_j = a_{j+2} = \cdots = a_{j+d-1}$.
Inductively we see that $R=(a_1\cdots a_{j-1})a_{j}^{\delta(n)-j+1}$.
Similarly, $a_1\cdots a_{j-1}a_j^{d-j+1}$ is in $\rho$ and $d-j+1 \pgq 2$. Again, there are no overlaps of $a_j^d$ with any element of $\rho\setminus \{a_j^d\}$ so $a_j = a_1 =  \cdots = a_{j-1}$.
Thus $R=a_{j}^{\delta(n)}$.

Now suppose that $j \pgq d$. Then $j-d+1 \pgq 1$, so by Consequence~\ref{csq:d-covering}, $a_{j-d+1}\cdots a_{j-1}a_j$ is in $\rho$.
As there are no overlaps of $a_j^d$ with any element of $\rho\setminus \{a_j^d\}$, it follows that $a_{j-d+1} = \dots = a_{j-1} = a_j$, and inductively $R = a_{j}^{j+1}(a_{j+2}\cdots a_{\delta(n)})$.
Using Consequence~\ref{csq:d-covering} again, $a_j^{d-1}a_{j+2}$ is in $\rho$ so $a_j = a_{j+2}$.
Inductively, we have $R=a_{j}^{\delta(n)}$.

Hence, for all $j$,
\[R = a_{j}^{\delta(n)} =
\begin{cases}(a^d_{j})^{(n/2)} &\mbox{if $n$ even}\\
(a^d_{j})^{((n-1)/2)}a_{j} &\mbox{if $n$ odd}.
\end{cases} \]

Let $x_i$ be the generator in $H$ corresponding to the loop $a_{j}$, so $1\ppq i\ppq u$ and $|x_i|=2$. Then $x_i$ acts on $E(\Lambda)$ as left multiplication by $a_{j}^d$.
Hence
\[R=\begin{cases}(x_i)^{(n/2)}\mo(a_{j}) &\mbox{if $n$ even}\\
(x_i)^{((n-1)/2)}a_{j} &\mbox{if $n$ odd} \end{cases} \]
with $x_i \in H$, $\mo(a_{j}) \in \R^0$ and $a_j\in\R^1$, so that $\mo(a_{j})$ and $a_j$ are in $\bigcup_{n=0}^N\R^n$.

\medskip

{\bf Case (2):} $k>1$. We note by our choice of $j,k$ that $a_{j}\cdots a_{j+k-1}$ is a closed trail of length $k$, which we denote by $T$. Let $\rho_T$ be the set of all paths of length $d$ which lie on $T$.

The first step is to show that $\rho_T$ is contained in $\rho$.
If $d \pgq k+1$, then this follows from Consequence~\ref{csq:trail}. So, suppose that $d \ppq k$.
Recall that
\[R=(a_1\cdots a_{j-1})(a_{j}\cdots a_{j+k-1})(a_{j}a_{j+k+1}\cdots a_{\delta(n)}).\]
Then: \vspace*{-\baselineskip}
\begin{align*}
&a_ja_{j+1}\cdots a_{j+d-1},\\
&a_{j+1}a_{j+2}\cdots a_{j+d},\\
&\qquad\vdots\\
& a_{j+k-d}a_{j+k-d+1}\cdots a_{j+k-1},\\
&a_{j+k-d+1}a_{j+k-d+2}\cdots a_{j+k-1}a_j
\end{align*}
are all paths of length $d$ which are subpaths of $R$, and so, by Consequence~\ref{csq:d-covering}, are in $\rho$.

Now $a_ja_{j+1}\cdots a_{j+d-1}$ overlaps $a_{j+k-d+1}a_{j+k-d+2}\cdots a_{j+k-1}a_j$.
So there is an element $R^2_1 \in \rho$ such that $R^2_1$ maximally overlaps $a_{j+k-d+1}a_{j+k-d+2}\cdots a_{j+k-1}a_j$ with maximal overlap of length $d+1$.
Then we have that
\[R^2_1 = a_{j+k-d+2}a_{j+k-d+3}\cdots a_{j+k-1}a_ja_{j+1}\]
and this maximal overlap is $\left (a_{j+k-d+1}a_{j+k-d+2}\cdots a_{j+k-1}a_j\right ) a_{j+1} = a_{j+k-d+1}R^2_1$.
Continuing in this way, $a_{j+1}a_{j+2}\cdots a_{j+d}$ overlaps $R^2_1$. So there is an element $R^2_2 \in \rho$ such that $R^2_2$ maximally overlaps $R^2_1$ with maximal overlap of length $d+1$.
So \[R^2_2 = a_{j+k-d+3}a_{j+k-d+4}\cdots a_{j+k-1}a_ja_{j+1}a_{j+2}\]
and this maximal overlap is $R^2_1a_{j+2} = a_{j+k-d+2}R^2_2$.
Inductively, we see that every path of length $d$ on the closed trail $T$ is in $\rho$.
Hence $\rho_T$ is contained in $\rho$.

It follows from Condition~\ref{cond:comb-d-koszul}~(2), that there are no paths in $\rho\setminus\rho_T$ which begin or end with any of the arrows $a_{j}, a_{j+1}, \dots , a_{j+k-1}$.

\medskip

Next we show that $R$ can be written in the form $R=p_1T^qp_2$, where $p_1$ is a suffix of $T$ and $p_2$ is a prefix of $T$.
If $d=2$ then $a_ja_{j+k+1}$ is a subpath of $R$ of length 2 and hence is in $\rho$.
By Condition~\ref{cond:comb-d-koszul}~(2) $a_ja_{j+k+1}$ must be in $\rho_T$ and so $a_{j+k+1}=a_{j+1}$.
Then $a_{j+k+1}a_{j+k+2} = a_{j+1}a_{j+k+2}$ and is a subpath of $R$ of length 2, so we must have that $a_{j+k+2}=a_{j+2}$.
Inductively, we see that $R$ lies on the closed trail $T$. So $R=p_1T^qp_2$, where $p_1$ is a suffix of $T$ and $p_2$ is a prefix of $T$.

So let $d\pgq 3$, and suppose first that $d \ppq k$. Then $a_{j+k-d+2}\cdots a_{j+k-1}a_ja_{j+k+1}$ is a subpath of $R$ of length $d$ which begins with the arrow $a_{j+k-d+2} \in \{ a_{j}, a_{j+1}, \dots , a_{j+k-1}\}$.
So, by Consequence~\ref{csq:d-covering} and Condition~\ref{cond:comb-d-koszul}~(2), this path is in $\rho_T$ and hence $a_{j+k+1} = a_{j+1}$.
Inductively, we have $a_{j+k+2}=a_{j+2},\ a_{j+k+3}=a_{j+3}, \dots$.
Similarly, $a_{j-1}a_j\cdots a_{j+d-2}$ is a
subpath of $R$ of length $d$ which ends with the arrow $a_{j+d-2} \in \{ a_{j}, a_{j+1}, \dots , a_{j+k-1}\}$.
So this path is in $\rho_T$ and hence $a_{j-1} = a_{j+k-1}$.
Inductively, we have $a_{j-2}=a_{j+k-2},\ a_{j-3}=a_{j+k-3}, \dots$.
So we may write $R=p_1T^qp_2$, where $p_1$ is a suffix of $T$ and $p_2$ is a prefix of $T$.

Now suppose that $d \pgq k+1$ (with $d\pgq 3$). We consider $j \ppq d-1$ and $j \pgq d$ separately.
Let $j \ppq d-1$. Then $j+d < \delta(n)$, so
$a_{j+1}a_{j+2}\cdots a_{j+k-1}a_ja_{j+k+1} \cdots a_{j+d}$ is a subpath of $R$ of length $d$ and starts with the arrow $a_{j+1} \in \{ a_{j}, a_{j+1}, \dots , a_{j+k-1}\}$.
So by Consequence~\ref{csq:d-covering} and Condition~\ref{cond:comb-d-koszul}~(2), this path is in $\rho_T$ and hence $a_{j+k+1} = a_{j+1}, a_{j+k+2} = a_{j+2}, \dots$.
So, inductively, we may write $R=(a_1 \cdots a_{j-1})T^qp_2$, where $p_2$ is a prefix of $T$.
Now $a_1a_2 \cdots a_{j-1} \cdots a_d$ is a subpath of $R$ of length $d$ and ends with the arrow $a_d \in \{ a_{j}, a_{j+1}, \dots , a_{j+k-1}\}$.
So by Condition~\ref{cond:comb-d-koszul}~(2), this path is in $\rho_T$ and hence $a_{j-1} = a_{j+k-1}, a_{j-2} = a_{j+k-2}, \dots$. Thus $R=p_1T^qp_2$, where $p_1$ is a suffix of $T$ and $p_2$ is a prefix of $T$.
Finally, suppose $j \pgq d$. Then, we know that $a_{j+k-d}\cdots a_{j-1}a_j \cdots a_{j+k-1}$ is a subpath of $R$ of length $d$ and ends with the arrow $a_{j+k-1} \in \{ a_{j}, a_{j+1}, \dots , a_{j+k-1}\}$.
So by Consequence~\ref{csq:d-covering} and Condition~\ref{cond:comb-d-koszul}~(2), this path is in $\rho_T$ and hence $a_{j-1} = a_{j+k-1}, a_{j-2} = a_{j+k-2}, \dots$.
Also $a_{j+k-d+2} \cdots a_{j-1}a_j \cdots a_{j+k+1}$ is a subpath of $R$ of length $d$ and starts with the arrow $a_{j+k-d+2}$. But we have just shown that $a_{j+k-d+2} \in \{ a_{j}, a_{j+1}, \dots , a_{j+k-1}\}$.
So again, this path is in $\rho_T$ and hence $a_{j+k+1} = a_{j+1}$. Inductively, $a_{j+k+2} = a_{j+2}, \dots$.
Thus $R=p_1T^qp_2$, where $p_1$ is a suffix of $T$ and $p_2$ is a prefix of $T$.

So, in all cases, $R=p_1T^qp_2$, where $T=a_{j}\cdots a_{j+k-1}$, $p_1$ is a suffix of $T$ and $p_2$ is a prefix of $T$.

\medskip

Without loss of generality, relabel the trail $T$ and write $R=T^qp$, where $T=a_1\cdots a_k$, $p$ is a prefix of $T$, $\delta(n)=kq+\ell(p)$, and we choose $\ell(p)$ in the range $1\ppq \ell(p)\ppq k$. Note that $R$ has a repeated arrow so $q \pgq 1$, and if $q=1$ then $\ell(p) \pgq 1$; moreover if $\ell(p) = k$ then $p=T$ and $R=T^{q+1}$.

Let $x_i$ be the generator in $H$ corresponding to this closed trail $T$, so $u+1 \ppq i \ppq r$. Let 
\begin{align*}
&T_{i,1}= T =a_1a_2\cdots a_k;\\
&T_{i,2}= a_2a_3\cdots a_ka_1;\\
&\qquad\vdots\\
&T_{i,k} =a_ka_1 \cdots a_{k-1}.
\end{align*}
The action of $x_i$ on $E(\Lambda)$ is left multiplication by
\[T_{i,1}^{d/\gcd(d,k)}+T_{i,2}^{d/\gcd(d,k)}+\dots+T_{i,k}^{d/\gcd(d,k)}\]
and $|x_i| = 2k/\gcd(d,k)$. Consequently, $N \pgq 2k/\gcd(d,k)$.
Now $R=T^qp$ with $1\ppq\ell(p)\ppq k$. Write $q = \frac{d}{\gcd(d,k)}c + w$ with $0 \ppq w \ppq \frac{d}{\gcd(d,k)}-1$.
Then
\[R=\left ( T^{d/\gcd(d,k)}\right )^{c}(T^wp).\]
Moreover, from the construction of $R$ as a maximal overlap, we see that $T^wp$ is also constructed as a maximal overlap and so corresponds to a basis element of $E(\Lambda)$.
We have $\ell(T^wp) = kw+\ell(p)\ppq k\left ( \frac{d}{\gcd(d,k)}-1\right ) + k = kd/\gcd(d,k) = \delta(2k/\gcd(d,k))$. So $T^wp$ corresponds to a basis element of $E(\Lambda)$ of degree at most $2k/\gcd(d,k)$, that is, $T^wp$ is in $\R^m$ for some $m\ppq N$.

Let $2\ppq l\ppq k$; we show that $T_{i,l}^{d/\gcd(d,k)}(T^wp)=0$ in $E(\Lambda)$.
We have $T_{i,l}=a_la_{l+1} \cdots a_ka_1\cdots a_{l-1}$, $T = a_1a_2\cdots a_k$ and $p=a_1a_2\cdots a_{\ell(p)}$ with $1 \ppq \ell(p) \ppq k$.
If $T_{i,l}^{d/\gcd(d,k)}(T^wp)$ represents a non-zero element in $E(\Lambda)$, then $\mt(a_{l-1}) = \mo(a_1)$ so that $a_1\cdots a_{l-1}$ is a closed trail.
But $l-1 < k$, so this contradicts the minimality of $k$. Hence $T_{i,l}^{d/\gcd(d,k)}(T^wp) = 0$ in $E(\Lambda)$ for $2\ppq l\ppq k$.

A similar argument also shows that \[\left ( \sum_{l=1}^k T_{i,l}^{d/\gcd(d,k)}\right )^{c} = \sum_{l=1}^k\left ( T_{i,l}^{d/\gcd(d,k)}\right )^{c}.\]
Thus \[R=\left ( T^{d/\gcd(d,k)}\right )^{c}(T^wp) = \sum_{l=1}^k\left ( T_{i,l}^{d/\gcd(d,k)}\right )^{c}(T^wp)= \left ( \sum_{l=1}^k T_{i,l}^{d/\gcd(d,k)}\right )^{c}(T^wp).\]
Hence $R = x_i^{c}(T^wp)$ with $x_i$ in $H$, and $T^wp \in \R^m$ for some $m\ppq N$.

\medskip

Hence each $R \in \R^n$ with $n > N$ can be written in the form $hr$ for some $h \in H$ and $r \in \bigcup_{n=0}^N\R^n$. It follows that $\bigcup_{n=0}^N\R^n$ is a generating set for $E(\Lambda)$ as a left $H$-module. Thus we conclude that $\Lambda$ has \rm\textbf{(Fg)}.
\end{proof}

\begin{example}
  We return to Example \ref{example:dKoszul}. Condition~\ref{cond:comb-d-koszul}
  is satisfied: the only closed trails that are not loops are the cycles of length $3$ (whose arrows are the $\gamma _i$); for all of these closed trails $T$, we have $\rho _T=\rho \setminus\set{\alpha ^3}$. Hence by Theorem~\ref{tm:J-sufficient} the algebra $\Lambda
  =K\cQ/I$ satisfies \textbf{(Fg)}.
\end{example}

\subsection{Conditions equivalent to \textbf{(Fg)} for a $d$-Koszul monomial algebra}\label{subsec:cond-equiv-d-Koszul}

Our aim is now to prove the converse, and more precisely, the following theorem.

\begin{thm}\label{tm:characterisations-fg-d-koszul}
  Let $\Lambda $ be an indecomposable finite dimensional $d$-Koszul monomial $K$-algebra with $d \pgq 2$. Consider the following statements:
  \begin{enumerate}
   \item[\ref{cond:fg}] $\Lambda $ satisfies \textbf{(Fg)};
  \item[\ref{cond:J-da-stacked}] Condition~\ref{cond:comb-d-koszul} holds for $\Lambda $;
  \item[\ref{cond:Zgr}] $Z_{\text{gr}}(E(\Lambda ))$ is Noetherian and $E(\Lambda)$ is a finitely generated  $Z_{\text{gr}}(E(\Lambda ))$-module;
 \item[\ref{cond:Zgr-light}] $E(\Lambda )$ is finitely generated as a module over $Z_{\text{gr}}(E(\Lambda ))$.
 \end{enumerate} 
 Then \ref{cond:Zgr-light} implies \ref{cond:J-da-stacked} which in turn implies \ref{cond:fg}. 

Moreover, if the field $K$ is algebraically closed, then the four statements are equivalent.
\end{thm}

We shall need the description of the Ext algebra $E(\Lambda )$ from \cite{GMMVZ}, which we recall here.

\medskip

Let $\cQ ^\op$ be the opposite quiver of $\cQ $, so that $\cQ ^\op_0=\cQ _0$ and $\cQ _1^\op=\set{\ov\alpha \colon j\rightarrow i\mid 
\text{there is } \alpha \colon i\rightarrow j\text{ in }\cQ _1}$. 

Now consider $\ulkd\Lambda =K\cQ ^\op/J$ with $J=(\lo\rho)$,  where the
orthogonal is taken with respect to the natural bilinear form $K\cQ ^\op_d\times K\cQ _d\rightarrow K$, that is, $\rep{\ov\beta _d\cdots \ov\beta _1,\alpha _1\cdots \alpha _d}$ is equal to $1$ if $\alpha _1\cdots \alpha _d=\beta _1\cdots \beta _d$ and is equal to $0$ otherwise. 
(Recall that, for any $n\pgq0$, $\cQ _n$ denotes the set of paths of length $n$ in $\cQ $.)

If $d=2$, set $B=\ulkd \Lambda $ and if $d\pgq 3$ let $B=\bigoplus _{n\pgq 0}B_n$ be the algebra defined as follows:
\begin{itemize}
\item  $B_n=\ulkd\Lambda_{\delta (n)} $
\item for $x\in B_n$ and $y\in B_m$, define $x\cdot y\in B_{n+m}$ by\[ x\cdot y=
\begin{cases}
0&\text{if $n$ and $m$ are odd}\\xy &\text{(in $\ulkd\Lambda $) if $n$ or $m$ is even.}
\end{cases} 
 \]
\end{itemize} Note that if $n$ or $m$ is even, $\delta (n)+\delta (m)=\delta (n+m)$, 
so that this defines a graded algebra $B$.

Then by \cite{gmv,bgs,GMMVZ}, the algebras $E(\Lambda )$ and $B$ are isomorphic
(for $d\pgq2$).

\bigskip

Since $\Lambda $ is monomial it is easy to see that the algebra $\ulkd\Lambda $
is $d$-homogeneous monomial and that the set $\sigma$ of paths $\ov\alpha _d\cdots\ov\alpha _1\in \cQ _d^\op$ such that $\alpha _1\cdots \alpha _d\in \cQ _d$ is not in $\rho $ is a minimal generating set for $J$ consisting of paths of length $d$. There is a basis $\B_{\ulkd\Lambda }$ of $\ulkd\Lambda $ consisting of all paths $\ov p$ in $\cQ ^\op$ such that no subpath of $\ov p$ is in $\sigma $.
It follows from Consequence \ref{csq:d-covering} that no subpath of length $d$ of $\ov R_i^n$ is in
  $\sigma $. Therefore $\ov R_i^n\in\B_{\ulkd\Lambda } $.

As we mentioned in Subsection \ref{subsec:Ext algebra}, there is a basis of $E(\Lambda )$ indexed by $\bigcup_{n\pgq 0}\R^n$ that corresponds, via the isomorphism with the algebra $B$, to the set of paths $\ov R_i^n$ for $n\pgq 0$ and $R_i^n\in\R^n$.
We then have an embedding of the basis $\bigcup_{n\pgq 0}\ov \R^n$ of $B$ into $\B_{\ulkd\Lambda }$ where $\ov \R^n=\set{\ov R_i^n\mid R_i^n\in\R^n}$; denote by $\B_B$ its image, which is a basis of  $B$.

\bigskip

We now define several gradings on $\Lambda $, $\ulkd \Lambda $ and $B$.

There are natural gradings on $\Lambda $ and on $\ulkd\Lambda $ given by the lengths of paths; denote the length by $\ell$ for both algebras. The degree of a homogeneous element $x$ in $B$ will be denoted by $\abs{x}$, so  $x\in \ulkd\Lambda _{\delta (\abs{x})}$ or, in other terms, $\abs x=k$ if, and only if, $\ell(x)=\delta (k)$.

The algebra $\ulkd\Lambda $ is also multi-graded by $\nn^{\cQ _1}$: for each path $\ov p$ in $\cQ ^\op$, we define an element $\mg(\ov p)=\left(\mg_\alpha (\ov p)\right)_{\alpha \in \cQ _1}\in \nn^{\cQ _1}$ as follows:
\begin{itemize}
\item if $\ell(\ov p)=0$, then $\mg(\ov p)=(0)_{\alpha \in \cQ _1}$
\item if $\ell(\ov p)>0$, then $\mg_\alpha (\ov p)$ is the number of occurrences of $\alpha $ in $p$ (it is $0$ if $\alpha $ does not occur in $p$).
\end{itemize}  Since $\ulkd\Lambda $ is monomial, 
the ideal $J$ is homogeneous with respect to this multi-degree and therefore $\ulkd\Lambda  $ is multi-graded. 
In $B$, if $x$ and $y$ are homogeneous and $\abs x$ or $\abs y$ is even, then $\mg_\alpha (xy)=\mg_\alpha (x)+\mg_\alpha (y)$ but if both degrees are odd then $\mg_\alpha (xy)=0$.

Let $Z:=Z_{\text{gr}}(B )$ be the graded centre of $B$. It is generated as a subring of $B$ by the homogeneous elements $z\in B$ such that for all homogeneous $y\in B$, $zy=(-1)^{\abs{y}\abs{z}}yz$. 
Note that $Z\subset \bigoplus_{e\in \cQ _0}eB e$, therefore $Z$ is generated by  elements that are linear combinations of (non-zero) cycles in $\cQ ^\op$.

Moreover, it can be checked easily that the graded centre $Z$ of $B$ is generated by elements $z$ that are  homogeneous with respect to the grading $\abs \cdot$ and the multi-degree $\mg$ and such that, for any element $y\in B$ that is homogeneous with respect to the grading $\abs\cdot$, we have $zy=(-1)^{\abs y\abs z}yz$.

\begin{remark}
If $d=2$ then $B\cong E(\Lambda )$ is generated in degrees $0$ and $1$, so in
order to check that an element of $B$ is in $Z$, it is sufficient to check that it is a
linear combination of cycles and  that it commutes or anti-commutes with all arrows in
$\cQ ^\op$. 

If $d\pgq 3$, then $B\cong E(\Lambda )$ is generated in degrees $0$, $1$ and
$2$. Therefore when checking that an element is in $Z$, we need to check that it
is a linear combination of cycles and  that it commutes or anti-commutes with paths of degrees $1$ and $2$, that is, arrows and (non-zero) paths of length $d$ in $\cQ ^\op $.
\end{remark}

The proof of Theorem \ref{tm:characterisations-fg-d-koszul} relies on some preliminary results. These are Lemma~\ref{lm:dKoszul-power-loop-Z}, Proposition~\ref{cj-d-trails} and Lemma~\ref{lm:dKoszul-power-closed-trail-Z}. For the first of these, we start with the following observation about loops. Suppose $\alpha $ is a loop in $\cQ_1$.  Since $\Lambda $ is finite dimensional, there
is some integer $N$ such that $\alpha ^N\in I$; therefore $\alpha ^N$ has a
subpath of length $d$ that is in $\rho $ and so $\alpha ^d\in \rho $.  Therefore $\ov\alpha^ d\not\in\sigma $ and it follows that $\ov\alpha ^j\neq 0$ in $\ulkd\Lambda $ for all $j\pgq 0$ and that $\ov\alpha ^{\delta (j)}\neq 0$ in $B$ for all $j\pgq 0$.

\begin{lem}\label{lm:dKoszul-power-loop-Z} Let $\alpha $ be a loop in $\cQ _1$ and
  let $n\pgq 2$. Then 
\begin{itemize}
\item if $d=2$,   $\ov\alpha ^{n}\in Z$ if, and only if,  $n$ is even
  and $\alpha $ satisfies  Condition~\ref{cond:comb-d-koszul}~(1);
\item if $d\pgq 3$,   $\ov\alpha ^{\delta (n)}\in Z$ if, and only if,  $\alpha $ satisfies   Condition~\ref{cond:comb-d-koszul}~(1).
\end{itemize}

\end{lem}

\begin{proof}

First note that if $n$ is odd, then 
\begin{itemize}
\item if $d=2$, $\ov\alpha ^n\ov \alpha =\ov\alpha ^{n+1}\neq (-1)^{n}\ov\alpha\,
  \ov\alpha ^n$, therefore $\ov\alpha ^{\delta (n)}=\ov\alpha ^n\not\in Z$;
\item if $d\pgq 3$, $\ov\alpha ^{\delta (n)}$ anti-commutes with all arrows since the products are $0$ in
    $B$.
\end{itemize} Therefore we may assume that $n\pgq 2$ is an even integer and that $\ov\alpha ^{\delta (n)}\in Z$.
 Set $e=\mo(\alpha )$.

 Let $\beta $ be an arrow ending at $e$ with $\beta \neq \alpha $. Then, in $B$, 
\[ \ov\alpha ^{\delta (n)+d-1}\ov\beta =\ov\alpha ^{\delta (n)}\cdot\ov\alpha
^{d-1}\ov\beta =\ov\alpha ^{d-1}\ov\beta \cdot\ov\alpha ^{\delta (n)}=\ov\alpha
^{d-1}\ov\beta \ov\alpha ^{\delta (n)} \] (the element $\ov\alpha ^{d-1}\ov\beta
$ is in degree $2$). Therefore we have an equality $ \ov\alpha ^{\delta (n)+d-1}\ov\beta
=\ov\alpha ^{d-1}\ov\beta \ov\alpha ^{\delta (n)}$ between two paths in the monomial algebra $\ulkd\Lambda $,  so that both paths are zero. In particular, $ \ov\alpha ^{\delta
  (n)+d-1}\ov\beta$ contains a subpath in $\sigma $, and since
$\ov\alpha^d\not\in\sigma  $, we must have $\ov\alpha ^{d-1}\ov\beta \in\sigma
$. It follows that $\beta \alpha ^{d-1}\not\in\rho $.

 Similarly, if $\beta $ is an arrow that starts at $e$ with $\beta \neq \alpha $, then $\alpha ^{d-1}\beta \not\in\rho $. 

 Therefore $\alpha $ satisfies Condition~\ref{cond:comb-d-koszul}~(1).

\medskip 

 Conversely
 , assume that  $\alpha $ satisfies Condition~\ref{cond:comb-d-koszul}~(1). 

Let $\beta \neq \alpha $ be an arrow and let $n\pgq 2$. By assumption, $\alpha ^{d-1}\beta
\not\in\rho $ and $\beta \alpha ^{d-1}\not\in\rho $. It follows that $\ov\beta
\ov\alpha ^{d-1}=0=\ov\alpha ^{d-1}\ov\beta $ in $\ulkd\Lambda $ (either in
$\sigma $ or not composable) and therefore that $\ov\alpha ^{\delta (n)}\ov\beta
=0=\ov\beta \ov\alpha ^{\delta (n)}$ since $\delta (n)\pgq\delta (2)\pgq d-1.$ The path $\ov\alpha ^{\delta (n)}$ anticommutes with all elements of degree $1$.

In particular, if $d=2$ and $n$ is even, then $\ov\alpha ^n\in Z$.

 Now assume that $d\pgq 3$ and consider commutation with elements in $B_2$. As a vector space, $B_2$ is generated by the paths $\ov p$ of length $d$ such that $p\in\rho $. Let $p=\beta _1\cdots\beta _d$ be a path in $\rho $. Since $\ov p$ has degree $2$ in $B$, products of elements in $B$ with $\ov p$ in $B$ and in $\ulkd\Lambda $ are equal. 

If $p=\alpha ^d$, then clearly $\ov\alpha ^{\delta (n)}\ov p=\ov p\,\ov \alpha ^{\delta (n)}$.

If $p\neq \alpha ^d$, set $j=\min\set{i\mid 1\ppq i\ppq d,\  \beta _i\neq \alpha }$ and  $k=\max\set{i\mid  1\ppq i\ppq d,\  \beta _i\neq \alpha }$.
By assumption, $\alpha ^{d-1}\beta _j\not\in\rho $ and $\beta _k\alpha ^{d-1}\not\in\rho $, therefore $\ov\beta _j\ov\alpha ^{d-1}=0$ and $\ov\alpha ^{d-1}\ov\beta _k=0$ in $\ulkd\Lambda $.  We have assumed that $n\pgq 2$, so $\delta (n)\pgq d$ and therefore $\ov \alpha ^{\delta (n)}\ov p=0=\ov p\,\ov\alpha ^{\delta (n)}$ in $\ulkd \Lambda $ and in $B$,  and $\ov\alpha ^{\delta (n)}$ anticommutes with elements of degree $2$.

Finally, we have proved that $\ov\alpha ^{\delta (n)}$ is in $ Z$ whenever $d=2$ and $n\pgq 2$ is even or $d\pgq 3$ and $n\pgq 2$.
\end{proof}

For the next result, we need some more terminology for closed trails. 
Let  $n\pgq 2$ and let $T=\alpha _1\cdots \alpha _n$ be a closed trail in $\cQ$ with $\rho_T=\set{\alpha _i\cdots \alpha _{i+d-1}\mid 1\ppq i\ppq n}\subseteq\rho $.
A \df{subcycle} of $T$ is a cycle of the form $q=\alpha _i\cdots \alpha _j$ with $1\ppq i\ppq j\ppq n$ and $\ell(q)<\ell(T)$. 
We say that $T$ has a \df{repeated vertex} if
$T = \alpha_1\cdots\alpha_{i-1} v\alpha_{i}\cdots \alpha_{i+k-1} v\alpha_{i+k}\cdots\alpha_n$
for some $i, k$ and vertex $v$ such that the paths $\alpha_{i}\cdots\alpha_{i+k-1}$ and $\alpha_{i+k}\cdots\alpha_n\alpha_1\cdots\alpha_{i-1}$
are non-trivial paths in $K{\mathcal \cQ }$.

We make the following assumptions that we use in Proposition~\ref{cj-d-trails} and Lemma~\ref{lm:dKoszul-power-closed-trail-Z}. The reason for these specific assumptions becomes clear in the proof of Theorem~\ref{tm:characterisations-fg-d-koszul}.
\begin{enumerate}[(i)]
\item none of the $\alpha _i$ are loops.
\item no subcycle of $T$ satisfies the same assumptions as $T$ (that is, there is no subcycle $q$ of $T$ of length at least $2$ with $\rho _q\subseteq\rho $). 
\end{enumerate}

  \begin{prop}\label{cj-d-trails}
Let $T=\alpha _1\cdots \alpha _n$ be a closed trail with $n\pgq 2$, $\rho
_T\subseteq\rho $ and such  that assumptions (i) and (ii) hold.    Let $p$ be a path of length $d$ such that $\mg_\beta (\ov p)=0$ if $\beta \not\in\set{\alpha _1,\ldots,\alpha _n}$ and which does not lie on $T$. Then $p\not\in\rho $.
  \end{prop}

  \begin{proof}

Let $p = \gamma_1\cdots \gamma_d$ with $\gamma_i \in \{\alpha_1, \dots , \alpha_n\}$ for $i=1, \dots, d$.
The path $p$ is a non-zero path in $K{ \cQ }$ so $\mathfrak{t}(\gamma_i) = \mathfrak{o}(\gamma_{i+1})$ for all $i$.
Suppose that $\gamma_1 = \alpha_j$ so $p = \alpha_j\gamma_2\cdots \gamma_d$ and $\gamma_2 \in \{\alpha_1, \dots , \alpha_n\}$ with $\mathfrak{t}(\alpha_j) = \mathfrak{o}(\alpha_{j+1}) = \mathfrak{o}(\gamma_2)$.
If $T$ does not have a repeated vertex then necessarily $\alpha_{j+1} = \gamma_2$. Inductively, $p = \alpha_j\alpha_{j+1}\cdots \alpha_{j+d-1}$ and hence $p$ lies on the trail $T$. This contradicts our hypothesis.
Hence $T$ has a repeated vertex.

Suppose that $v$ is a repeated vertex so that $T$ has a proper subpath $q$ of length $k$ with $q = vqv$ for some $k$; thus $2 \leqslant k \leqslant n-1$ since $T$ does not have any loops and $q$ is a closed trail. We claim that $k\geqslant d$. Indeed, if  we had $k<d$, then by Consequence \ref{csq:trail} every path of length $d$ that lies on $q$ would be in $\rho $, that is $\rho _q\subseteq\rho $, with $\ell(q)<\ell(T)$. But this contradicts assumption (ii).
Hence $k \geqslant d$.

Now suppose for contradiction that $p\in\rho$. As above, let $\alpha_j$ be the first arrow in $p$. We know that $p$ does not lie on $T$ and that $T$ has a repeated vertex, so we may write
\[p = \alpha_j\cdots\alpha_{j+r-1}\gamma_{r+1}\cdots \gamma_d\]
where $r \geqslant 1$, $\gamma_{r+1}, \dots, \gamma_d \in \{\alpha_1, \dots , \alpha_n\}$,
$\mathfrak{t}(\alpha_{j+r-1}) = \mathfrak{o}(\alpha_{j+r}) = \mathfrak{o}(\gamma_{r+1})$ and
$\gamma_{r+1}\neq \alpha_{j+r}$. Then there is some $t$ such that $\gamma_{r+1}= \alpha_{t}$ and $t\not\equiv j+r\pmod n$.
Moreover $\mathfrak{t}(\alpha_{t-1}) = \mathfrak{o}(\alpha_t) = \mathfrak{o}(\alpha_{j+r})$. We may illustrate this as follows:
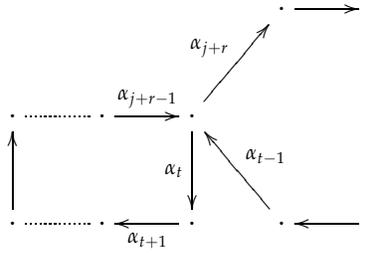
\[\xymatrix{
 & & & \cdot\ar[r] & \cdot\\
\cdot\ar@{.}[r] & \cdot\ar[r]^{\alpha_{j+r-1}} & \cdot\ar[ur]^{\alpha_{j+r}}\ar[d]_{\alpha_t} & & \vdots\\
\cdot\ar[u]\ar@{.}[r] & \cdot & \cdot\ar[l]^{\alpha_{t+1}} & \cdot\ar[ul]_{\alpha_{t-1}} & \cdot\ar[l]
}\]
(We make no assumption as to whether $\gamma_{r+2}$ is or is not equal to $\alpha_{t+1}$.)
We note that the path $\alpha_{j+r}\cdots\alpha_{t-1}\cdot \alpha_t\cdots \alpha_{j+r-1}$ is a cyclic permutation of $T$ and has length $n$. Moreover, from the previous part of this proof, both $\alpha_{j+r}\cdots\alpha_{t-1}$ and $\alpha_t\cdots \alpha_{j+r-1}$ are paths of length at least $d$.

Let $S =\alpha_{t}\alpha_{t+1} \cdots \alpha_{j+r-1}$. Then $S$ is a closed path in $K{\mathcal \cQ }$ of length at least 2 and is a subcycle of $T$.
There is an overlap $\alpha_{j+r-d}\cdots\alpha_j \cdots\alpha_{j+r-1}$ with $p$ so the subpath $\alpha_{j+r-d+1}\cdots\alpha_{j+r-1}\alpha_{t}$ must also be in $\rho$.
Then we have an overlap of $\alpha_{j+r-d+1}\cdots\alpha_{j+r-1}\alpha_{t}$ with
$\alpha_t\alpha_{t+1}\cdots \alpha_{t+d-1}$; by Property~\ref{pty:d-covering}, all subpaths of length $d$ of the path
\[\alpha_{j+r-d+1}\cdots\alpha_{j+r-1}\alpha_t\alpha_{t+1}\cdots \alpha_{t+d-1}\] must also be in $\rho$.
Thus every path of length $d$ that lies on $S$ is in $\rho$.
Hence $\rho_S \subseteq \rho$.

So $S$ is a subcycle of $T$ that satisfies the same assumptions as $T$, and this contradicts assumption~(ii). Hence $p \not\in \rho$.
  \end{proof}

 \begin{remark}\label{rk:da-trails} We keep the assumptions and notation of Proposition~\ref{cj-d-trails}. 
 Then $\ov T$ and all the paths lying on $\ov T$ are in the basis $\B_{\ulkd\Lambda }$
since none of their subpaths of length $d$ are in $\sigma $; those of length
$\delta (k)$ for some $k\pgq0$ are in the basis $\B_B$. 
  
 Set $T_i=\alpha _i\cdots \alpha _n\alpha _1\cdots\alpha _{i-1}$ so that $\ov T_i=\ov\alpha _{i-1}\cdots\ov\alpha_ 1\ov\alpha _n\cdots \ov\alpha _i\in\ulkd\Lambda $ for all $i$. 
 Then for any $j\pgq 1$, we have
\begin{align*}
\ov T_i^j\ov\alpha _k\neq 0&\iff k=i-1\\
\ov \alpha _k\ov T_i^j\neq 0&\iff k=i.
\end{align*}
\end{remark}

\begin{lem}\label{lm:dKoszul-power-closed-trail-Z}
 Let $T=\alpha _1\cdots\alpha _n$ be a closed trail with $n\pgq 2$ and $\rho
 _T\subseteq\rho $ that satisfies assumptions~(i) and (ii)
 and set $z_j=\sum_{i=1}^n\ov T_i^j$ with $nj=\delta (u)$ for some integer $u\pgq2$
 and with $nj=u$ even if $d=2$. Then   $z_j\in Z$ if, and only if,  $T $ satisfies Condition~\ref{cond:comb-d-koszul}~(2).

Moreover, if $T $ does not satisfy  Condition~\ref{cond:comb-d-koszul}~(2), then no element in $ B$ that is  homogeneous with respect to $\abs{\cdot}$ and $\mg$ (when viewed in $\ulkd\Lambda $) and that is a linear combination of non-trivial cycles that lie on $\ov T$ is in $Z$.
\end{lem}

\begin{proof}
 First assume that $z_j\in Z$. Fix an integer $i$ and let $e=\mt(\alpha _i)$. Suppose for contradiction that there is a path $p$  of length $d-1$
  starting at $e$ such that $\alpha _i p \in\rho $ and
  $ p \neq \alpha _{i+1}\cdots\alpha _{i+d-1}$. By Remark \ref{rk:da-trails}, at least one arrow in $ p $ is not in
  $\set{\alpha _1,\ldots,\alpha _n}$, therefore $\ov p\,\ov\alpha _i  $ is not a subpath of any of the paths that occur in $z_j$ (note that $\ell(z_j)\pgq \delta (2)=d\pgq\ell(\ov p)$). The relation
  $\ov{\alpha _i p }z_j=z_j\ov{\alpha _i p }$ in $B$ becomes, in $\ulkd\Lambda $,
  \[ \ov p\, \ov\alpha _i\ov T_i^j=z_j\ov p \,\ov\alpha _i.\]
  Since $\ulkd\Lambda $ is monomial, it follows that
  $\ov p\,\ov \alpha _i\ov T_i^j$ contains a subpath in $\sigma $, that is, there is a subpath of length $d$ of $T_i^j\alpha _ip$ that is not in $\rho $. It cannot be a subpath of $T_i^j\alpha _i$ because we have assumed that $\rho _T\subseteq\rho $. Moreover, we have assumed that $\alpha _ ip  \in \rho $. We also have $\alpha _{i-d+1}\cdots \alpha _{i-1}\alpha _i\in\rho _T\subseteq \rho $ and $\rho$ is $d$-covering (Property~\ref{pty:d-covering}), therefore every subpath of length $d$ of $\alpha _{i-d+1}\cdots \alpha _{i-1}\alpha _i p$ is also in $\rho $ and so is every subpath of length $d$ of $T_i^j\alpha _i p$ and we  have obtained a contradiction. Therefore $T$ satisfies Condition~\ref{cond:comb-d-koszul}~(2).

 Conversely
 , assume that $T$ satisfies Condition~\ref{cond:comb-d-koszul}~(2). We prove that for all $j\pgq 1$ such that $nj=\delta (u)$ for
  some integer $u\pgq 2$, and such that $nj=u$ is even if $d=2$, we have $z_j\in Z$.

  \begin{itemize}
  \item First note that if $d\pgq 3$ and $\abs{z_j}$ is odd, then $z_j$ anti-commutes
    with all arrows (the products are $0$ in $B$). Therefore assume that
    $\abs{z_j}$ is even and that $d\pgq2$,   and let $\beta $ be an arrow.

    If $\beta =\alpha _k$ for some $k$, then
    $\ov\alpha _kz_j=\ov\alpha _k\ov T_k^j$ and
    $z_j\ov\alpha _k=\ov T_{k+1}^j\alpha _k$ using Remark \ref{rk:da-trails}, and these paths are indeed equal, so that
    $\ov\beta z_j=z_j\ov\beta =(-1)^{\abs{z_j}\abs{\ov\beta }}z_j\ov \beta $.

    If $\beta \not\in\set{\alpha _1,\ldots,\alpha _n}$ then
    $\ov \beta \,\ov T_i^j=0=\ov T_i^j\ov \beta $ for all $i$ by assumption,
    therefore $\ov \beta z_j=0=(-1)^{\abs{z_j}\abs{\ov\beta }}z_j\ov\beta $.
  \item Now assume that $d\pgq 3$ (and $\abs {z_j}$ is still even) 
  and consider commutation with elements in
    $B_2$. As a vector space, $B_2$ is generated by the paths $\ov p$ of length
    $d$ such that $p\in\rho $. Let $p=\beta _1\cdots \beta _d$ be a path in
    $\rho $ with $\beta _i\in \cQ _1$ for all $i$. Since $\ov p$ has degree $2$ in
    $B$, products of elements in $B$ with $\ov p$ in $B$ and in $\ulkd\Lambda $
    are equal.

    If $p$ lies on $T$, then $p=\alpha _k\cdots \alpha _{k+d-1}$ for some $k$,
    and it is easy to check that
    $\ov pz_j=z_j\ov p=(-1)^{\abs{z_j}\abs {\ov p}}z_j\ov p$ (as for commutation
    with an arrow).

    If $p$ does not lie on $T$, then by  Condition~\ref{cond:comb-d-koszul}~(2), the
    first and last arrows in $p$ are not in $\set{\alpha _1,\ldots,\alpha
      _n}$.
    Moreover, for all $i$, we have $\alpha _i\beta _1\cdots \beta _{d-1}\not\in \rho $;
    it follows that $\ov p z_j=0$. Similarly, for all $i$, we have
    $\beta _2\cdots \beta _d\alpha _i \not\in \rho $ (since it ends with $\alpha _i$ and is not in $\rho _T$), therefore $z_j\ov
    p=0$. Finally, $\ov p z_j=z_j\ov p=(-1)^{\abs{z_j}\abs {\ov p}}z_j\ov p$.

  \end{itemize} We have proved that $z_j\in Z$.

\medskip

Now  let $z$ be an element in $Z$ that is homogeneous with respect to $\abs{\cdot}$
  and $\mg$ and that is a linear combination of cycles that lie on $\ov
  T$.
  Assume that $\ell(z)>0$; by (i)  $z$ is not a linear combination of arrows so $\abs{z}\pgq 2$. Put
  $z=\sum_{i=1}^m\lambda _ic_i$ with $\lambda _i\in K$ and the $c_i$ cycles in
  $\cQ ^{\op}$ that lie on $\ov T$. Since $z$ is homogeneous with respect to
  $\abs{\cdot}$ and $\mg$ and the $c_i$ lie on $\ov T$, the $c_i$ are cyclic permutations of $c_1$. Up to
  relabelling, we may write
  $c_i=\ov T_i^{j-1}\ov \alpha _{i-1}\cdots \ov \alpha _{i-s}$ for some fixed
  $s$ with $1\ppq s\ppq n$ (and $m=n$).

  We first consider the case where $\abs{z}$ is even. Then we must have
  $\ov\alpha _k z=z\ov\alpha _k$ for all $k$, that is,
  $\sum_{i=1}^n\lambda _i\ov\alpha _k\ov T_i^{j-1}\ov\alpha
  _{i-1}\cdots\ov\alpha _{i-s}=\sum_{i=1}^n\lambda _i\ov T_i^{j-1}\ov\alpha
  _{i-1}\cdots\ov\alpha _{i-s}\ov\alpha _k$.
  Using Remark \ref{rk:da-trails}, this is equivalent to
  \[ \lambda _k\ov \alpha _k\ov T_k^{j-1}\ov\alpha _{k-1}\cdots\ov\alpha
  _{k-s}=\lambda _{k+s+1}\ov T_{k+s+1}^{j-1}\ov\alpha _{k+s}\cdots\ov\alpha
  _{k+1}\ov\alpha _k.\]
  Therefore $\lambda _k=0=\lambda _{k+s+1}$ or $k+s\equiv k\pmod{n}$ (that is,
  $s=n$), and $\lambda _k=\lambda _{k+1}$.

  This is true for all $k$, so if $z\neq 0$ then $z=\lambda
  _1z_j$ 
  with $nj=\ell(z)=\delta (\abs z)$.

  We now consider the case where $\abs{z}$ is odd.

  If $d=2$ then the same reasoning as in the even case shows that
  $\lambda _{k+1}=(-1)^{\abs z}\lambda _k$ for all $k$ with $1\ppq k\ppq n-1$ and
  $\lambda _1=(-1)^{\abs z}\lambda _n$, therefore $\lambda _1=(-1)^{n\abs z}\lambda _1=-\lambda _1$ (because $nj=\ell(z)=\abs z$ is odd hence $n$ is odd) so that $\lambda _k=0$ for all $k$ and
  finally $z=0$.

  Now assume that $d\pgq 3$.  Since $z\in Z$, we \sloppy have
  $\ov\alpha _{k+d-1}\cdots\ov\alpha _k z=z\ov\alpha _{k+d-1}\cdots\ov \alpha
  _k$
  for all $k$, that is,
  $\sum_{i=1}^n\lambda _i\ov\alpha _{k+d-1}\cdots\ov\alpha _k\ov
  T_i^{j-1}\ov\alpha _{i-1}\cdots\ov\alpha _{i-s}=\sum_{i=1}^n\lambda _i\ov
  T_i^{j-1}\ov\alpha _{i-1}\cdots\ov\alpha _{i-s}\ov\alpha
  _{k+d-1}\cdots\ov\alpha _k$.
  Using Remark \ref{rk:da-trails}, this is equivalent to
  \[ \lambda _k\ov\alpha _{k+d-1}\cdots\ov \alpha _k\ov T_k^{j-1}\ov\alpha
  _{k-1}\cdots\ov\alpha _{k-s}=\lambda _{k+s+d}\ov T_{k+s+d}^{j-1}\ov\alpha
  _{k+s+d-1}\cdots\ov\alpha _{k+d}\ov\alpha _{k+d-1}\cdots\ov\alpha _k.\]
  Therefore $\lambda _k=0=\lambda _{k+s+d}$ or $k+s+d-1\equiv k+d-1\pmod{n}$
  (that is, $s=n$), and $\lambda _k=\lambda _{k+d}$.

  When $s=n$, we have $nj=\ell(z)=\delta (\abs z)=\dfrac{\abs z-1}{2}d+1$,
  therefore $n$ and $d$ are coprime. It follows that if $z\neq 0$, then all the
  $\lambda _i$ are equal so that $z=\lambda _1z_j$.

  We have proved that if $z$ is a non-zero element in $Z$ that is homogeneous
  with respect to $\abs{\cdot}$ and $\mg$ and which is a linear combination of non-trivial 
  cycles that lie on $\ov T$, then if $d=2$ we must have $\abs z$ even and for
  all $d\pgq 2$, $z$ is then a non-zero scalar multiple of $z_j$. Therefore $z_j$ is in $Z$ and by the first part of the proof, $T$ satisfies Condition~\ref{cond:comb-d-koszul}~(2).
\end{proof}

We have now all the tools we need for the proof of Theorem~\ref{tm:characterisations-fg-d-koszul}.

\begin{proof}[Proof of Theorem \ref{tm:characterisations-fg-d-koszul}]
  The fact that \ref{cond:J-da-stacked} implies \ref{cond:fg} is Theorem~\ref{tm:J-sufficient}. 
  The implication \ref{cond:Zgr} $\Rightarrow$ \ref{cond:Zgr-light} is clear and if, in addition, $K$ is algebraically closed, then the implication \ref{cond:fg} $\Rightarrow$ \ref{cond:Zgr} follows from \cite{ES}. 

We now prove that \ref{cond:Zgr-light} implies \ref{cond:J-da-stacked}.
Suppose that \ref{cond:J-da-stacked} does not hold, that is, Condition~\ref{cond:comb-d-koszul} does not hold. Assume for contradiction that $B$ is a finitely generated $Z$-module, generated by elements $\ov g_1,\ldots,\ov g_t$ that are homogeneous with respect to both the  grading $\abs{\cdot}$ and the multi-grading $\mg$.

\bigskip

We first  assume that  Condition~\ref{cond:comb-d-koszul}~(1)  does not hold, so that  there is a
loop $\alpha $ that does not satisfy this condition. Then for all $j\pgq 2$, $\ov\alpha ^{\delta (j)}\not \in Z$ by Lemma \ref{lm:dKoszul-power-loop-Z}.

Now consider $\ov\alpha ^{\delta (k)}\in B$ for some even integer $k\pgq2$.  Then  
\[\mg_\beta (\ov\alpha ^{\delta (k)})=
\begin{cases}
\delta (k)&\text{ if } \beta =\alpha \\0&\text{ if } \beta \neq \alpha 
\end{cases}\]
and $\abs{\ov\alpha ^{\delta (k)}}=k$.
By assumption, and using the fact that $\ov\alpha ^{\delta (k)}$, $\ov g_1$, ..., $\ov g_t$ are homogeneous with respect to $\abs{\cdot}$ and $\mg$, there exist elements $\ov u_{i}^{(k)}$ in $Z$, $1\ppq i\ppq t$, that are homogeneous with respect to $\abs{\cdot}$ and $\mg$, such that $\ov\alpha ^{\delta (k)}=\sum_{i=1}^t\ov u_{i}^{(k)}\ov g_i$. 

Since $\ov\alpha ^{\delta (k)}$ is  homogeneous with respect to $\abs{\cdot}$ and $\mg$, we can assume that for all $i$ we have $\abs{\ov u_{i}^{(k)}\ov g_i}=k$ and $\mg(\ov u_{i}^{(k)})+\mg(\ov g_i)=\mg(\ov u_{i}^{(k)}\ov g_i)=\mg(\ov\alpha ^{\delta (k)})$.  
If $\beta \neq \alpha $ then $\mg_\beta (\ov u_{i}^{(k)})+\mg_\beta (\ov g_i)=0$ so $\mg_\beta (\ov u_{i}^{(k)})=0$ and $\mg_\beta (\ov g_i)=0$. It follows that $\ov u_i^{(k)}$ and $\ov g_i$ are  powers of $\ov\alpha $; since $\ov u_i^{(k)}\in Z$, we must have $\abs{u_i^{(k)}}=0$ or $1$ by assumption. If $\abs {\ov u_i^{(k)}}=1$, then $\abs{\ov g_i}=k-1$ is odd and hence, in $B$, we have $\ov u_i^{(k)}\ov g_i=0$. Therefore we may assume that   $\ov u_{i}^{(k)}\in Z^0=K$. 
It follows that the sum contains only one term and that $\ov g_i$ is a (non-zero) scalar multiple of $\ov\alpha ^{\delta (k)}$ so that $\ov\alpha ^{\delta (k)}\in\spn_ K\set{\ov g_1,\ldots,\ov g_t}$.

We have shown that $\set{\ov\alpha ^{\delta (k)}\mid k\pgq 1,\ k\text{ even}}\subseteq\spn_ K\set{\ov g_1,\ldots,\ov g_t}$. However, using the grading $\abs{\cdot}$, we see that the $\ov\alpha ^{\delta (k)}$, $k\pgq 1$, are linearly independent over $K$: we have reached a contradiction.

Therefore the Yoneda algebra $E(\Lambda )=B$ is not finitely generated as a
$Z$-module when Condition~\ref{cond:comb-d-koszul}~(1) does not hold.

\medskip

We now assume that Condition~\ref{cond:comb-d-koszul}~(2)  does not hold, so that there is a closed trail $T=\alpha _1\cdots\alpha _n $ with $n\pgq 2$ and $\rho _T\subseteq \rho $ that does not satisfy this condition. We can make the following assumptions (and therefore use Proposition \ref{cj-d-trails}):
\begin{enumerate}[(i)]
\item none of the $\alpha _i$ are loops. Indeed, if $\alpha _i$ is a loop, then
  the paths $\alpha _i\alpha _{i+1}\cdots \alpha _{i+d-1} $ and $\alpha _i^d$
  are in $\rho $ and properly overlap, hence $\alpha _i^{d-1}\alpha _{i+1}$
  is in $\rho $ because $\rho$ is $d$-covering, therefore $\alpha $ does not
  satisfy  Condition~\ref{cond:comb-d-koszul}~(1), and in this case we already know that $E(\Lambda )$ is not a finitely generated $Z$-module.
\item no subcycle of $T$ satisfies the same hypotheses as $T$ (otherwise replace $T$ with the shortest such subcycle).

\end{enumerate}

We have seen in Lemma \ref{lm:dKoszul-power-closed-trail-Z} that no linear combination of non-trivial cycles that lie on $\ov T$, that is homogeneous with respect to $\abs{\cdot}$ and $\mg$,  is in $Z$.

Now consider $z_{\delta (k)}=\sum_{i=1}^n\ov T_i^{\delta (k)}\in B$ for some
even integer $k\pgq 2$.  Then $\abs{z_{\delta (k)}}=nk$ and \sloppy $\mg_\beta (z_{\delta (k)})=
\begin{cases}
{\delta (k)}&\text{ if } \beta \in\set{\alpha _1,\ldots,\alpha _n}\\0&\text{ if } \beta \not\in\set{\alpha _1,\ldots,\alpha _n}.
\end{cases}
$

Since $z_{\delta (k)}$ and the $\ov g_i$ are  homogeneous with respect to $\abs{\cdot}$ and $\mg$, 
by assumption there exist elements $\ov u_{i}^{(k)}$ in $Z$, $1\ppq i\ppq t$, that are homogeneous with respect to $\abs{\cdot}$ and $\mg$, such that $z_{\delta (k)}=\sum_{i=1}^t\ov u_{i}^{(k)}\ov g_i$. Note that $Z\subseteq\bigoplus_{e\in \cQ _0}e\ulkd\Lambda e$, therefore the $\ov u_i^{(k)}$ are linear combinations of cycles.

Fix an integer $j$ with $1\ppq j\ppq n$. Then $\ov\alpha _{j+d-1}\cdots\ov \alpha _j\ov T_j^{\delta (k)}=\ov\alpha _{j+d-1}\cdots\ov\alpha _j\ov z_{\delta (k)}=\sum_{i=1}^t\ov\alpha _{j+d-1}\cdots\ov\alpha _j \ov u_{i}^{(k)}\ov g_i$.

Since $\ov\alpha _{j+d-1}\cdots\ov \alpha _j\ov T_j^{\delta (k)}$ is  homogeneous with respect to $\abs{\cdot}$ and $\mg$, we can assume that for all $i$ we have $2+\abs{\ov u_{i}^{(k)}}+\abs{\ov g_i}=\abs{\ov\alpha _{j+d-1}\cdots\ov\alpha _j\ov u_{i}^{(k)}\ov g_i}=\abs{\ov\alpha _{j+d-1}\cdots \ov \alpha _j z_{\delta (j)}}= nk+2$ and $\mg(\ov\alpha _{j+d-1}\cdots\ov \alpha _j)+\mg(\ov u_{i}^{(k)})+\mg(\ov g_i)=\mg(\ov\alpha _{j+d-1}\cdots\alpha _j\ov u_{i}^{(k)}\ov g_i)=\mg(\ov\alpha _{j+d-1}\cdots\ov \alpha _j\ov T_i^{\delta (k)})=\mg(\ov\alpha _{j+d-1}\cdots\ov \alpha _j)+\mg(\ov T_i^{\delta (k)})$, and therefore the only arrows that occur in $\ov u_{i}^{(k)}\ov g_i$ are the $\ov\alpha _j$, $1\ppq j\ppq n$. Moreover, Proposition \ref{cj-d-trails} and the fact that the $\ov u_i^{(k)}$ are linear combinations of cycles show that the $\ov u_i^{(k)}$ must be linear combinations of cycles lying on $\ov T$ (otherwise, one at least of these cycles has a subpath of length $d$ that does not lie on $\ov T$, hence that is in $\sigma $ and this cycle vanishes in $B$ and does not occur in $u_i^{(k)}$).  Our assumption shows that these cycles must be trivial (of length $0$), and since  $Z^0=K$ we see that $\ov u_i^{(k)}\in K$. Therefore $\ov\alpha _{j+d-1}\cdots\ov\alpha _j\ov T_j^{\delta (k)}$ is a linear combination of the $\ov\alpha _{j+d-1}\cdots\ov \alpha _j\ov g_i$, hence $\set{\ov\alpha _{j+d-1}\cdots\ov\alpha _j\ov T_j^{\delta (k)}\mid 1\ppq j\ppq n,\ k\pgq 1,\ k\text{ even}}\subseteq\spn_K\set{\ov\alpha _{j+d-1}\cdots\ov\alpha _j\ov g_i\mid 1\ppq j\ppq n,\ 1\ppq i\ppq t}$.

The set $\set{\ov\alpha _{j+d-1}\cdots\ov\alpha _{J}\ov T_{J}^{\delta (k)}\mid 1\ppq j\ppq n,\ k\pgq 2,\ k\text{ even}}$ is linearly independent over $K$ (using the  grading $\abs{\cdot}$), therefore we have reached a contradiction.

Therefore the Yoneda algebra $E(\Lambda )=B$ is not finitely generated as a
$Z$-module when Condition~\ref{cond:comb-d-koszul}~(2) does not hold. Hence \ref{cond:Zgr-light} $\Rightarrow$ \ref{cond:J-da-stacked}.
\end{proof}

\begin{remark} In the case where $K$ is algebraically closed, we have extended the equivalence between \ref{cond:fg} and \ref{cond:Zgr}, already known for Koszul algebras from \cite{ES}, to  $d$-Koszul monomial algebras with $d\pgq 3$.
In particular, we have the following corollary.
\end{remark}

\begin{corollary}\label{cor:Zgr-light}
  Let $\Lambda $ be a $d$-Koszul monomial algebra over an algebraically closed field with $d\pgq 2$. Assume that
  $E(\Lambda )$ is a finitely generated $Z_{\text{gr}}(E(\Lambda
  ))$-module. Then the algebra $Z_{\text{gr}}(E(\Lambda ))$ is Noetherian.
\end{corollary}

\section{Extension to $(D,A)$-stacked monomial algebras}\label{sec:da-stacked}

\subsection{Notation and properties of $(D,A)$-stacked monomial algebras}\label{subsec:da-stacked}

Let  $\Lambda=K\cQ /I $ be a monomial
algebra with the length grading as before.
Let $D$ and  $A$ be  integers with $D>A\pgq1$. From \cite[Definition~3.1]{GS-colloq math}, $\Lambda $ is then a $(D,A)$-stacked monomial algebra if, for any minimal projective right $\Lambda $-module resolution of $\Lambda_0$, the $n$-th projective module is generated in degree $\delta _A(n)$ where 
\[ \delta _A(n)=
\begin{cases}
n&\text{if $n=0$ or $n=1$}\\
\frac{n}{2}D& \text{if $n\pgq 2$ is even}\\
\frac{n-1}{2}D+A&\text{if $n\pgq 3$ is odd}.
\end{cases}
 \]
When $A=1$, we retrieve the definition of a $D$-Koszul algebra, so that a
$(D,1)$-stacked monomial algebra is a $D$-Koszul monomial algebra.  

It was shown in \cite[Proposition 3.3]{GS-colloq math} that if $\gldim{\Lambda }\pgq 4$ then
$A$ divides $ D$; in particular, $D\pgq 2A$. 
If  the global dimension of $\Lambda $ is
finite, then Condition \textbf{(Fg)}, and in fact all the conditions \ref{cond:fg}--\ref{cond:gorenstein} stated in the introduction, are satisfied by $\Lambda $. Therefore
we shall assume  throughout this section that $\Lambda$ is a $(D,A)$-stacked monomial algebra with $\gldim \Lambda \pgq 4$ and set $d=\frac{D}{A}$. 
We define 
\[ \delta  (n)=
\begin{cases}
n&\text{if $n=0$ or $n=1$}\\
\dfrac{n}{2}d=\dfrac{\delta _A(n)}A& \text{if $n\pgq 2$ is even}\\
\dfrac{n-1}{2}d+1=\dfrac{\delta _A(n)}{A}&\text{if $n\pgq 3$ is odd}.
\end{cases} \]

\begin{dfn} For $A\pgq 1$, we define an \df{\apath} as a non-zero path $p=\alpha _1\cdots \alpha _n$ where all the $\alpha_i $ are paths of length $A$ (that is, $\alpha_i \in \cQ_A$ for all $i$). An \df{$A$-trail} is an \apath in which all the $\alpha _i$ are distinct. An \df{\acycle} is a closed \apath  and finally an \df{\aloop} is an $A$-cycle of length $A$.

Given an \apath $p$ as above, an \df{\asubpath} of $p$ is an \apath of the form
$\alpha _i\cdots \alpha _j$ with $1\ppq i\ppq j\ppq n$ (note that not every \apath
that is a subpath of $p$ is an \asubpath of $p$). An \df{\asubcycle} of $p$ is
a closed \asubpath of one of the non-zero \apaths  $\alpha _i\cdots\alpha _n\alpha _1\cdots
\alpha _{i-1}$ with $1\ppq i\ppq n$.

We also define the \df{\alength} $\ell_A(p)$ of an \apath $p=\alpha _1\cdots \alpha _n$ where the $\alpha _i$ are paths of length $A$ as $\ell_A(p)=n$, that is, $\ell(p)=A\ell_A(p)$. 
\end{dfn}

We will need the following result from  \cite{GS-colloq math}.

\begin{pty}\cite[Section 3]{GS-colloq math}
\label{pty:da-stacked-covering}
  Let $\Lambda =K\cQ /I$ be a finite dimensional monomial algebra. Then $\Lambda $ is $(D,A)$-stacked  if,
  and only if, $\rho =\R^2$ has the following properties:
\begin{enumerate}[(1)]
\item every path in $\rho $ is of length $D$;
\item if $R_2^2\in\R^2$ properly overlaps $R_1^2\in \R^2$ with overlap $R_1^2u$, then $\ell(u)\pgq A$ and there exists $R_3^2\in \R^2$ which properly overlaps $R_1^2$ with overlap $R_1^2u'$, $\ell(u')=A$ and $u'$ is a prefix of $u$.
\end{enumerate}
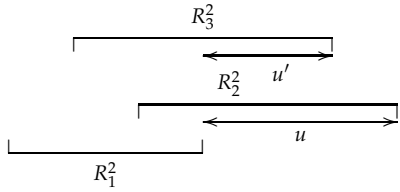
\[\xymatrix@W=0pt@M=0.3pt{
\ar@{^{|}-^{|}}@<-2.25ex>[rrr]_{R_1^2}
&\ar@{_{|}-_{|}}@<6ex>[rrrr]^{R_3^2}&
\ar@{_{|}-_{|}}@<1.25ex>[rrrr]^(.35){R_2^2} & \ar@{{<}-{>}}[rrr]_{u} \ar@{{<}-{>}}@<4.75ex>[rr]_(.6){u'} & &
& }\]
\end{pty}

 Therefore $\rho $ consists of paths of length $D$, and if $\Lambda $ is
 $(D,A)$-stacked with $\gldim\Lambda \pgq4$,  we view $\rho $ as a set of
 \apaths of \alength $d$.

 \begin{example}
   \label{example:DAstacked-FS}
We include first an example from \cite{FS} (Example 3.2).
Let $\Lambda = K\cQ/I$ where $\cQ$ is the quiver
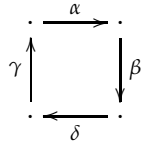
\[\xymatrix{
\cdot \ar[r]^{\alpha} & \cdot \ar[d]^{\beta}  \\
\cdot\ar[u]^{\gamma} & \cdot \ar[l]^{\delta}
}\]
and the ideal $I$ has minimal generating set 
$\rho = \{\alpha\beta\gamma\delta\alpha\beta, \gamma\delta\alpha\beta\gamma\delta\}.$
Then $\Lambda$ is a $(6,2)$-stacked monomial algebra.

The closed $2$-trails are all the paths of length $4$.
 \end{example}

 \begin{example}
   \label{example:DAstacked-stretched-dKoszul}
Now  we  give an example where, as well as closed \atrails, there are \aloops.
Let $\Lambda = K\cQ/I$ where $\cQ$ is the quiver
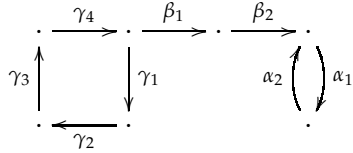
\[\xymatrix{
\cdot \ar[r]^{\gamma_4} & \cdot \ar[r]^{\beta_1}\ar[d]^{\gamma_1} & \cdot \ar[r]^{\beta_2} & \cdot\ar@/^/[d]^{\alpha_1} \\
\cdot\ar[u]^{\gamma_3} & \cdot \ar[l]^{\gamma_2} & & \cdot\ar@/^/[u]^{\alpha_2} 
}\]
and the ideal $I$ has minimal generating set 
$\rho = \{(\alpha_1\alpha_2)^2,  (\gamma_1\gamma_2)(\gamma_3\gamma_4),  
(\gamma_3\gamma_4)(\gamma_1\gamma_2)\}$.
Then $\Lambda$ is a $(4,2)$-stacked monomial algebra. 

The closed $2$-trails are the paths of length $4$ whose arrows are the $\gamma _i$ and
the $2$-loops are $\alpha _1\alpha _2$ and $\alpha _2\alpha _1$.

 \end{example}

 \begin{example}
   \label{example:DAstacked-not-stretched}
Finally, we give an example in which an arrow, namely $\beta _2$, occurs both in
closed \atrails and in \aloops.
Let $\Lambda = K\cQ/I$ where $\cQ$ is the quiver
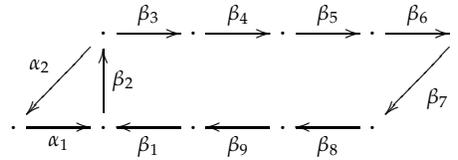
\[\xymatrix{
& \cdot \ar[r]^{\beta_3}\ar[dl]_{\alpha_2} & \cdot \ar[r]^{\beta_4} & \cdot \ar[r]^{\beta_5} & \cdot \ar[r]^{\beta_6} & \cdot \ar[dl]^{\beta_7} \\
\cdot\ar[r]_{\alpha_1} & \cdot \ar[u]_{\beta_2} &  \cdot \ar[l]^{\beta_1}  &  \cdot \ar[l]^{\beta_9}  &  \cdot \ar[l]^{\beta_8} &
}\]
and the ideal $I$ has minimal generating set 
\[\rho = \{(\alpha_1\beta_2\alpha_2)^2,  (\beta_1\beta_2\beta_3)(\beta_4\beta_5\beta_6), 
 (\beta_4\beta_5\beta_6)(\beta_7\beta_8\beta_9),  (\beta_7\beta_8\beta_9)(\beta_1\beta_2\beta_3)\}.\]
Then $\Lambda$ is a $(6,3)$-stacked monomial algebra. 

The closed $3$-trails are all the cycles of length $9$ whose arrows are the
$\beta _i$ and the $3$-loops are $\alpha _1\beta _2\alpha _2$, $\alpha _2\alpha
_1\beta _2$ and $\beta _2\alpha _2\alpha _1$.
 \end{example}

We have the following consequences of Property~\ref{pty:da-stacked-covering}.

\begin{csq} \label{csq:subpaths-d-a-covering}
We keep the notation of Property~\ref{pty:da-stacked-covering}, with $D=dA$. Then the length
of $u$ must be a multiple of $A$, so that $R_1^2u$ is an \apath, and every
\asubpath of \alength $d$ of $R_1^2u$ is in $\rho $. Moreover, no other subpath of length $D$ of
$R_1^2u$ is in $\rho $.
 \end{csq}

 \begin{proof} Write $\ell(u)=qA+r$ with $q\pgq 1$ and $0\ppq r<A$. We prove
   the result by induction on $q$.

If $q=1$, then the path $R_2^2\in\rho $ overlaps $R_3^2\in \rho $ with overlap $R_3^2u_3$
   for some path $u_3$. If this overlap is a proper overlap (that is, $R_3^2\neq
   R_2^2$), then $\ell(u)=\ell(u')+\ell(u_3)=A+\ell(u_3)$ so
   that $\ell(u_3)=(q-1)A+r=r$ and $0<r<A$. Therefore by Property
   \ref{pty:da-stacked-covering} we have a contradiction. It follows that
   $R_3^2=R_2^2$ and $u=u'$ has length $A$ and that $R_1^2$ and $R_3^2$ are the
   only \asubpaths of \alength $d$ of $R_1^2u$ and they are in $\rho
   $. Moreover, any other subpath of length $D$
   of $R_1^2u$ is a proper overlap of $R_1^2$ of length strictly smaller than $D+A$,
   which is impossible  by Property
   \ref{pty:da-stacked-covering}. 

Let $q>1$ be such that $\ell(u)=qA+r$ with $0\ppq r<A$ and assume that the
result is true for any proper overlap of a path in $\rho $ of length  $D+q'A+r'$ with $q'<q$ and
$0\ppq r'<A$. 
   The path $R_2^2\in\rho $ properly overlaps $R_3^2\in \rho $ with overlap $R_3^2u_3$
   for some path $u_3$ with $\ell(u)=\ell(u')+\ell(u_3)=A+\ell(u_3)$ so
   that $\ell(u_3)=(q-1)A+r$ and the overlap $R_3^2u_3$ has length $D+(q-1)A+r$. By induction, $\ell(u_3)$ is a multiple of $A$,
   therefore $r=0$ and $\ell(u)$ is a multiple of $A$. Any \asubpath of
   \alength $d$ of $R_1^2u$ is either $R_1^2$ or an \asubpath of \alength $d$ of
   $R_3^2u_3$. Again by induction, they are all in $\rho $. Finally, a subpath
   of length $D$ of $R_1^2u$ which is not an \asubpath  either is a subpath of
   length $D$ of $R_3^2u_3$ that is not an \asubpath, therefore not in $\rho $
   by induction, or  properly overlaps $R_1^2$ with overlap $R_1^2u'''$ with
   $0<\ell(u''')<A$, which is impossible by Property
   \ref{pty:da-stacked-covering}. 
 \end{proof}

\begin{csq}\label{csq:da-covering}
Suppose that $D=dA$. Let $n\pgq 2$ and let $R_i^n$ be an element of $\R^n$. Write $R_i^n=\alpha _1\cdots \alpha _{\delta  (n)}$ where each $\alpha _i$ is a path of length $A$.  
Then for all $i$ with $1\ppq i\ppq \delta (n)-d+1$, the path $\alpha _i\cdots \alpha _{i+d-1}$ is in $\rho $, that is, all the \asubpaths of \alength $d$ of $R_i^n$ are in $\rho $. Moreover, no other subpath of $R_i^n$ of length $D$ is in $\rho $. 
\end{csq}

\begin{proof}

The result is proved by induction.
It is clear when $n=2$. Moreover, if $n=3$,
since $R_i^3\in\R^3$ is a maximal overlap of two elements in $\R^2$, it follows
from Property \ref{pty:da-stacked-covering} and using the notation therein that
$R_i^3=R_1^2u'=v'R_3^2$ where $v'$ is the prefix of $R_1^2$ of length $A$. By
Consequence \ref{csq:subpaths-d-a-covering}, the only subpaths of length $D$ of
$R_i^3$ that are in $\rho =\R^2$ are $R_1^2$ and $R_3^2$.

Now let $n\pgq 4$ and take $R_i^n\in \R^n$. Then $R_i^n$ is a
maximal overlap of $R_1^2\in\R^2$ with $R_2^{n-1}\in \R^{n-1}$ so that
$R_i^n=R_2^{n-1}u$ for some path $u$. Write
$R_2^{n-1}=\alpha _1\cdots \alpha _{\delta (n-1)}$ with $\ell(\alpha _i)=A$ for
all $i$. By the   induction assumption, we have $\alpha _i\cdots \alpha _{i+d-1}\in
\R^2$ for all $i$ with $i+d-1\ppq \delta (n-1)$.  In particular, $R_3^2:=\alpha _{\delta (n-1)-d+1}\cdots \alpha
_{\delta (n-1)}$ is in $\R^2$. Since $R_1^2$ overlaps $R_3^2$ with overlap
$R_3^2u$, by Property ~\ref{pty:da-stacked-covering} we have
$\ell(u)=A$ and $\alpha _{\delta (n-1)-d+2}\cdots \alpha _{\delta(n-1)}u=R_1^2\in\R^2$. Since $R_i^n=R_2^{n-1}u$, we have proved the first part of the result for
$R_i^n$.

Now let $p$ be another subpath of $R_i^n$ of length $D$. We already know by
induction that if $p$ is a subpath of $R_2^{n-1}$, then $p$ is not in
$\R^2$. Therefore $p$ is a subpath of $R_3^2u$ which is neither
$R_3^2$ nor $R_1^2$. 
By Consequence \ref{csq:subpaths-d-a-covering}, $p$ is not in $\rho $. 
We have proved that $p\not\in
\R^2$ and the induction step is complete.  
\end{proof}

\begin{csq}\label{csq:atrail}
Suppose that $D=dA$. Let $T = \alpha_1 \cdots \alpha_n$ be a closed \atrail in ${\cQ}$ with $\alpha _i\in\cQ_A$ for all $i$ and suppose that $d \pgq n+1$. Assume also that $T$ is the prefix of an \apath in $\rho $ and the suffix of an \apath in $\rho $.
Then all \asubpaths of \alength $d$ of powers of  the closed trail $T$ are in $\rho$.
\end{csq}

\begin{proof} By assumption, there exist \apaths $T'$ and $T''$ such that $T'T\in \rho $ and $TT''\in \rho $. 
Since $\Lambda$ is finite dimensional, there is a path $R_2 \in \rho$ that lies
on $T$, and $\ell(R_2)=D=dA>\ell(T)=nA$. Therefore $R_2$ is a subpath of length
$D$ of $T^N=(\alpha _1\cdots \alpha _n)^N$ for some $N\pgq 2$. If $R_2=T^m$ is a
power of $T$ with $m\pgq 2$ (and $d=nm$) then $R_2$ overlaps itself with overlap
$T^{2m-1}$ and the result follows using Consequence
\ref{csq:subpaths-d-a-covering} (every \asubpath of \alength $d$ of a power of
$T$ is an \asubpath of $T^{2m-1}$). Otherwise, $TT''$ overlaps $R_2$ or $R_2$
overlaps $T'T$ and we can use Consequence \ref{csq:subpaths-d-a-covering} again
to prove that $R_2$ is an \asubpath of  $T^N$ and then that every \asubpath of
\alength $d$ of the overlap is in $\rho $; since every \asubpath of \alength $d$
of a power of $T$ is one of these, we obtain the result.
\end{proof}

\subsection{Characterisations of $(D,A)$-stacked monomial algebras that satisfy \textbf{(Fg)}}

We now give our combinatorial condition 
for $(D,A)$-stacked monomial algebras $\Lambda$.

\begin{cond}\label{cond:comb-da}  We say that a  $(D,A)$-stacked monomial algebra $\Lambda $ satisfies Condition \ref{cond:comb-da}, or \ref{cond:J-da-stacked},  when the following properties (1) and (2) both hold:
\begin{enumerate}[(1)]
\item Let $c$ be an \aloop in $\cQ _A$. Write $c=a_1\cdots a_A$ with $a_i\in \cQ _1$ for all $i$ and $c_j=a_j\cdots a_Aa_1\cdots a_{j-1}$ for $j\in\set{1,\ldots,A}$. Then there exists $j$ such that $c_j^d\in\rho $ but there is no path in $\rho $ of the form $c_j^{d-1}\beta $ or $\beta c_j^{d-1}$ where $\beta $ is a path of length $A$ that is distinct from $c_j$.
\item Let $T=\alpha _1\cdots\alpha _n$ be a closed $A$-trail in $\cQ $ with $n\pgq 2$ and $\alpha _i\in \cQ _A$ for all $i$ and such that $\rho _T:=\set{\alpha _1\cdots\alpha _d,\alpha _2\cdots\alpha _d\alpha _{d+1},\ldots ,\alpha _n\alpha _1\cdots\alpha _{d-1}}\subseteq\rho $. Then there are no elements in $\rho \setminus\rho _T$ which begin or end with the path $\alpha _i$, for all $i$.
\end{enumerate}
\end{cond}

\begin{remark}
  In part (1) of the condition, there is exactly one $j$ such that $c_j^d\in \rho $. Indeed, if $c_j^d$ and $c_k^d$ were in $\rho $, they would overlap with an overlap of length at most $D+A-1$, hence by Property \ref{pty:da-stacked-covering} we must have $c_j^d=c_k^d$ and therefore $j=k$.   
\end{remark}

\begin{remark}
  If $A=1$ then Condition~\ref{cond:comb-da} is equivalent to Condition~\ref{cond:comb-d-koszul}.
\end{remark}

We first prove that this condition is sufficient for $\Lambda $ to satisfy \textbf{(Fg)}.

\begin{thm}\label{tm:da-stacked-sufficient}
Let $\Lambda = K\cQ/I$ be a finite dimensional $(D,A)$-stacked
monomial algebra. Assume that $\Lambda$ satisfies
Condition~\ref{cond:comb-da}. Then $\Lambda$ satisfies \textbf{(Fg)}.
\end{thm}

\begin{proof}
The case $D\pgq 2$ and $A=1$ corresponds to $d$-Koszul monomial algebras (with
$D=d$) and is proved in Theorem~\ref{tm:J-sufficient}. Therefore we assume that $A > 1$ so that necessarily $D > 2$. 
If  $\gldim\Lambda$ is finite then $\Lambda$ satisfies \textbf{(Fg)} (and Condition~\ref{cond:comb-da} is empty), so we also assume that  $\gldim\Lambda \pgq 4$ so that $D=dA$. 

The structure of this proof follows that of Theorem~\ref{tm:J-sufficient} by replacing each arrow in $\cQ_1$ by a path of length $A$ in $\cQ_A$. We do not give all the details here, but indicate those places where we need to provide additional arguments.

\medskip

The first part of the proof is to show that the hypotheses of \cite[Theorem 3.4]{GS-colloq math} hold.

Let $c_1, \dots , c_u$ be the $A$-loops in $\cQ$ such that $c_i^d \in \rho$ for $i=1, \dots , u$.
(We remark that, in the terminology of \cite{GS-colloq math}, these are precisely the closed paths in $\cQ$ such that for each $c_i$ we have $c_i \neq p_i^{r_i}$ for any path $p_i$ with $r_i \pgq 2$ and $c_i^d \in \rho$.
Firstly, $c_i^d \in \rho$ implies that $\ell(c_i) = A$. Then, if $c_i = p_i^{r_i}$ for some path $p_i$ with $r_i \pgq 2$, we have $1 \ppq \ell(p_i) < A$.
Now $p_i^{dr_i}$ is in $\rho$ and $p_i^{dr_i}$ overlaps itself with overlap $p_i^{dr_i+1}$, so there is a maximal overlap in $\R^3$ of length $\ppq D+\ell(p_i)<D+A$.
But this is a contradiction since $\Lambda$ is a $(D,A)$-stacked monomial algebra. So $c_i \neq p_i^{r_i}$.)
By Condition~\ref{cond:comb-da}~(1), for each $i=1, \dots, u$, there are no elements in $\rho$ of the form $c_i^{d-1}\beta$ or $\beta c_i^{d-1}$ where $\beta$ is a path of length $A$ that is distinct from $c_i$.

We need to show that there are no overlaps of $c_i^d$ with any element of $\rho\setminus\{c_i^d\}$.
If $R \in \rho\setminus\{c_i^d\}$ and $R$ overlaps $c_i^d$, then, by  Consequence~\ref{csq:subpaths-d-a-covering}, either $R = c_i^sb$ or $R = bc_i^s$ where
$1 \ppq s \ppq d-1$ and $b$ is an \apath with $\ell_A(b) = d-s$ and that does not begin (respectively, end) with the path $c_i$.
Suppose that $R = c_i^sb$. Then $R$ overlaps $c_i^d$ with overlap of length $A(2d-s)$.
By Consequence~\ref{csq:subpaths-d-a-covering}, this is a maximal overlap since $c_i$ is not a prefix of $b$ and thus gives an element $R^3_1 \in {\mathcal R}^3$. However, $\ell(R^3_1) = D+A = (d+1)A$. Thus $2d-s = d+1$ and so $s=d-1$.
But then $R = c_i^{d-1}b$ and $b$ is a path of length $A$ distinct from $c_i$, which is a contradiction. The case  $R = bc_i^s$ is similar.
So there are no overlaps of $c_i^d$ with any element of $\rho\setminus\{c_i^d\}$.

Let $T_{u+1}, \dots T_r$ be the distinct closed $A$-trails in $\cQ$ with $\ell_A(T_i) > 1$ such that the sets $\rho_{T_i}$ of Condition~\ref{cond:comb-da}~(2) are contained in $\rho$.
For each $i=u+1, \dots,r$,
we write $T_i=\alpha_{i,1}\cdots\alpha_{i,m_i}$, where the $\alpha_{i,j}$ are in $\cQ_A$ so that
$\ell_A(T_i) = m_i > 1$ and
\[\rho_{T_i}=\{\alpha_{i,1}\cdots\alpha_{i,d}, \alpha_{i,2}\cdots\alpha_{i,d+1},\dots,  \alpha_{i,m_i}\alpha_{i,1}\cdots\alpha_{i,d-1}\} \subseteq \rho.\]
By Condition~\ref{cond:comb-da}~(2), for each closed $A$-trail $T_i$ ($i=u+1, \dots,r$), there are no elements in $\rho\setminus\rho_{T_i}$ which begin or end with the path $\alpha_{i,j}$, for all $j = 1, \dots, m_i$. So no path $\alpha_{i,j}$ of length $A$ has overlaps with any element in $\rho\setminus\rho_{T_i}$.

\medskip

The next step is to describe a  commutative Noetherian graded subalgebra $H$ of $\HH^*(\Lambda )$ with $H^0=\HH^0(\Lambda )$. 
Applying \cite[Theorem 3.4]{GS-colloq math}, gives
$\HH^*(\Lambda)/\mathcal{N}\cong K[x_1,\dots,x_r]/\langle x_ax_b \ \mbox{for} \ a\neq b\rangle$, where
\begin{itemize}
\item for $i=1,\dots,u$, the vertices $v_1,\dots,v_u$ are distinct and the element $x_i$ corresponding to the $A$-loop $c_i$ is in degree $2$ and is represented by the map ${\mathcal P}^2\longrightarrow\Lambda$ where for $R^2\in\mathcal{R}^2$,
    \[\mo(R^2)\otimes\mt(R^2)\mapsto \begin{cases} v_i & \mbox{ if}\ R^2=c_i^d \\
    0 & \mbox{ otherwise}
    \end{cases}\]
\item and for $i=u+1,\dots,r$, the element $x_i$ corresponding to the closed $A$-trail $T_i=\alpha_{i,1}\cdots\alpha_{i,m_i}$ is in degree $2\mu_i$ such that $\mu_i=m_i/\gcd(d,m_i)$ and is represented by the map ${\mathcal P}^{2\mu_i}\longrightarrow\Lambda$, where for $R^{2\mu_i}\in{\mathcal{R}{^{2\mu_i}}}$,
    \[\mo(R^{2\mu_i})\otimes\mt(R^{2\mu_i})\mapsto
\begin{cases}\mo(T_{i,k}) & \mbox{if}\ R^{2\mu_i}=T^{d/\gcd(d,m_i)}_{i,k}
\mbox{ for all } k=1,\dots,m_i  \\
0 &\mbox{otherwise}.
\end{cases}\]
\end{itemize}
Let $H$ be the subring of $\HH^*(\Lambda)$ generated by $Z(\Lambda)$ and $\{x_1,\dots,x_r\}$.
As in Theorem~\ref{tm:J-sufficient}, $H$ is a commutative Noetherian ring.

\medskip

Now we show that $\Lambda$ satisfies \rm\textbf{(Fg)} with this algebra $H$.
Again, we identify $\bigcup_{n\pgq 0}\R^n$ with a basis of $E(\Lambda)$.
Set $N=\max\{3, |x_1|,\dots,|x_r|, |\mathcal{Q}_A|\}$.
We show that $\bigcup_{n=0}^N\R^n$ is a generating
set for $E(\Lambda)$ as a left $H$-module and thus $E(\Lambda)$ is finitely generated as a left $H$-module.

Let $R\in \R^n$ with $n>N$. Then $\ell_A(R) = \delta(n) \pgq 2d$
and we can write $R=a_1a_2\cdots a_{\delta(n)}$ where the $a_i$ are in $\cQ_A$.
The proof now follows that of Theorem~\ref{tm:J-sufficient} by replacing each arrow by a path of length $A$, and with extensive use of Consequences~\ref{csq:subpaths-d-a-covering}, \ref{csq:da-covering} and \ref{csq:atrail}, and Condition~\ref{cond:comb-da}. 
Thus we conclude that $\Lambda$ has \rm\textbf{(Fg)}.
\end{proof}

\begin{example}
  We return to Examples \ref{example:DAstacked-FS},
  \ref{example:DAstacked-stretched-dKoszul} and \ref{example:DAstacked-not-stretched}.
In all these examples, Condition \ref{cond:comb-da} is satisfied and therefore
\textbf{(Fg)} holds for $\Lambda $. 

For instance, in Example \ref{example:DAstacked-FS}, the only closed $2$-trails $T$
such that $\rho _T\subseteq\rho $ are $\alpha \beta \gamma \delta $ and $\gamma
\delta \alpha \beta $ and, in both cases, $\rho _T=\rho $. In Example
\ref{example:DAstacked-not-stretched}, the closed $3$-trails  $T$
such that $\rho _T\subseteq\rho $ are those that start with $\beta _1$, $\beta
_4$ and $\beta _7$, in all cases we have $\rho _T=\rho \setminus\set{(\alpha
  _1\beta _2\alpha _2)^2}$ and $(\alpha
  _1\beta _2\alpha _2)^2$ does not start or end with a $\beta _i$.

By \cite[Theorem 2.5]{EHSST}, it follows that $\Lambda $ is Gorenstein in each
case. Moreover, it was proved in \cite{FS} that the algebra in Example
\ref{example:DAstacked-FS} has injective dimension $2$.
\end{example}

Our aim is now to prove the following theorem, and in particular the converse of Theorem \ref{tm:da-stacked-sufficient}.

\begin{thm}\label{tm:characterisations-fg-da-stacked}
  Let $\Lambda $ be an indecomposable finite dimensional $(D,A)$-stacked
  monomial algebra. Suppose that $D\neq 2A$ whenever $A>1$. Consider the following statements:
  \begin{enumerate}
  \item[\ref{cond:fg}] $\Lambda $ satisfies \textbf{(Fg)};
  \item[\ref{cond:J-da-stacked}] Condition~\ref{cond:comb-da} holds for $\Lambda $;
  \item[\ref{cond:Zgr}] $Z_{\text{gr}}(E(\Lambda ))$ is Noetherian and $E(\Lambda)$ is a finitely generated  $Z_{\text{gr}}(E(\Lambda ))$-module;
  \item[\ref{cond:Zgr-light}]  $E(\Lambda )$ is finitely generated as a module over $Z_{\text{gr}}(E(\Lambda ))$.
 \end{enumerate}  Then \ref{cond:Zgr-light} implies \ref{cond:J-da-stacked} which in turn implies \ref{cond:fg}. 

Moreover, if the field $K$ is algebraically closed, then the four statements are equivalent. 
\end{thm}

We shall need, as in the $d$-Koszul case, a description of the Ext algebra of
$\Lambda $. We give the details of this in the appendix, and we briefly describe
it here. Since we have already proved
Theorem~\ref{tm:characterisations-fg-da-stacked} when $\Lambda $ is $d$-Koszul,
we assume here that $D>A>1$ and, in addition, that $D\neq 2A$.

Let $\Gamma $ be the quiver with the same vertices as $\cQ $ and whose set of arrows
corresponds to the set of paths of length $A$ in $\cQ$, that is, $\Gamma
_1=\set{\ov\alpha \colon i\rightarrow j\mid \text{there exists } \alpha \in \cQ _A, \ \alpha \colon j\rightarrow i}$. Let $\lo \rho $ be the orthogonal of $\rho $ for the bilinear form $K\Gamma _d\times K(\cQ _A)_d\rightarrow K $ defined on paths of length $d$ in $\Gamma $ and \apaths of \alength $d $ in $\cQ $ by $\rep{\ov\alpha _d\cdots\ov\alpha _1,\beta _1\cdots \beta _d}=1$ if $\alpha _1\cdots\alpha _d=\beta _1\cdots\beta _d$ and $0$ otherwise, where the $\alpha _i$ and $\beta _i$ are in $\cQ _A$. Set $\lkd\Lambda =K\Gamma /J$ where $J=(\lo\rho )$; it is a monomial algebra and the ideal $J$ has a minimal generating set $\sigma $ given by all the paths $\ov \alpha _d\cdots\ov\alpha _1$ such that the \apath $\alpha _1\cdots\alpha _d$ is not in $\rho $.

Let $B=\bigoplus _{n\pgq 0}B_n$ be the algebra defined as follows:
\begin{itemize}
\item  $B_n=\lkd\Lambda_{\delta  (n)} $
\item for $x\in B_n$ and $y\in B_m$, define $x\cdot y\in B_{m+n}$ by
\[ x\cdot y=
\begin{cases}
0&\text{ if $n$ and $m$ are odd};\\
0&\text{ if $n$ or $m$ is equal to $1$ and $n\pgq 1$, $m\pgq 1$};\\
xy&\text{ in $\lkd\Lambda $ otherwise}. 
\end{cases}
 \] 
\end{itemize}
Observe that if $n$ or $m$ is even and both are larger than $1$, $\delta (n)+\delta  (m)=\delta  (n+m)$, so that the algebra $B$ is a graded $K$-algebra, generated in degrees $0$, $1$, $2$ and $3$. Note that this is also true of $E(\Lambda )$ by \cite{GS-J Alg}.  Moreover, we prove in the appendix that the algebras $E(\Lambda )$ and $B$ are isomorphic, generalising the description given in \cite{GMMVZ} when $\Lambda $ is a  $d$-Koszul  algebra. 
{This isomorphism uses the assumption that $D \neq 2A$.}

There is a basis $\B_{\lkd\Lambda }$ of $\lkd\Lambda $ consisting of all paths $p$ in $\Gamma $ such that no path in $\sigma $ is a subpath of $p$, and basis $\B_B$ of $B$ contained in $\B_{\lkd\Lambda }$ consisting of all $\ov R_i^m$ for all $m\pgq0$ and all $R_i^m\in\R^m$.

We now define several gradings, on $\lkd \Lambda $ and on $B$.

There is a natural grading on $\lkd\Lambda $  given by the length $\ell$ of paths. Note that if $p$ is an \apath in $\cQ $, then $\ell(\ov p)=\ell_A(p)$. The degree of a homogeneous element $x$ in $B$ will be denoted by $\abs{x}$, so  $x\in \lkd\Lambda _{\delta  (\abs{x})}$ or, in other terms, $\abs x=k$ if, and only if, $\ell(x)=\delta (k)$.

The algebra $\lkd\Lambda $ is also multi-graded by $\nn^{\cQ _1}$: for each path $\ov p$ in $\Gamma $, we define an element $\mg(\ov p)=\left(\mg_\alpha (\ov p)\right)_{\alpha \in \cQ _A}\in \nn^{\cQ _1}$ as follows: write the \apath $p$ in $\cQ $ as $p=\alpha _1\cdots \alpha _n$ where each $\alpha _i$ is in $\cQ _A$; 
\begin{itemize}
\item if $\ell(\ov p)=0$, then $\mg(\ov p)=(0)_{\alpha \in \cQ _1}$
\item if $\ell(\ov p)>0$, then $\mg_\alpha (\ov p)$ is the number of $\alpha _i$ that are equal to $\alpha $  (it is $0$ if none of the $\alpha _i$ are equal to $\alpha $).
\end{itemize}  Note that even if $\alpha $ is a subpath of $p$, we can have $\mg_\alpha (\ov p)=0$ (if $\alpha $ is not one of the $\alpha _i$, that is, $p=q\alpha r$ where $q$ and $r$ are paths in $\cQ $ whose lengths are not multiples of $A$). 

Since $\lkd\Lambda $ is monomial, the ideal $J$ is homogeneous with respect to this multi-degree and therefore $\lkd\Lambda  $ is multi-graded.  In $B$, if $x$ and $y$ are homogeneous and $\abs x$ or $\abs y$ is even with both degrees at least $2$, then $\mg_\alpha (xy)=\mg_\alpha (x)+\mg_\alpha (y)$ but  $\mg_\alpha (xy)=0$ otherwise. 

Let $Z:=Z_{\text{gr}}(B)$ be the graded centre of $B$. As in the $d$-Koszul case, it   is generated by elements $z$ that are  homogeneous with respect to the grading $\abs \cdot$ and the multi-degree and such that, for any element $y\in B$ that is homogeneous with respect to the  grading $\abs\cdot$, we have $zy=(-1)^{\abs y\abs z}yz$.

\begin{remark}
\label{rk:Z-da-stacked-degrees}
  Recall that $B\cong E(\Lambda )$ is generated in degrees $0$, $1$, $2$ and $3$
  and that the product of an element of degree $1$ with any other element
  vanishes. Therefore when checking that an element is in $Z$, we need to check
  that it is a linear combination of cycles and  that it commutes or anti-commutes with paths of degrees $2$ and $3$, that is,  (non-zero) paths of length $d$ and of length $d+1$ in $\Gamma $.
\end{remark}

\medskip

The proof of Theorem~\ref{tm:characterisations-fg-da-stacked} relies on some preliminary results, namely Lemma~\ref{lm:da-stacked-power-aloop-Z}, Proposition~\ref{cj-da-trails} and Lemma~\ref{lm:d-a-stacked-power-closed-trail-Z}. We start with some comments on  $A$-loops in 
$\cQ _A$.   Let $c $ be an \aloop in $\cQ _A$. Since $\Lambda $ is finite dimensional, there
  exists an integer $N$ such that $c^N=0$ in $\Lambda $ and therefore there is
  some $j$ such that $c_j^d\in \rho $. To simplify notation and without loss of
  generality, write $c=c_j$. Then $c ^d\in\rho $, therefore $\ov c^ d\not\in\sigma $ and it follows that $\ov c ^k\neq 0$ in $\lkd\Lambda $ for all $k\pgq 0$ and that $\ov c ^{\delta (k)}\neq 0$ in $B$ for all $k\pgq 0$.

\begin{lem}\label{lm:da-stacked-power-aloop-Z} Let $c$ be an \aloop in $\cQ _A$ and let $n\pgq 2$ be an integer. Then 
  $\ov c^{\delta (n)}\in Z$ if, and only if,  $c $ satisfies  Condition~\ref{cond:comb-da}~(1).
\end{lem}

\begin{proof}
The proof is very similar to that of Lemma \ref{lm:dKoszul-power-loop-Z}, using $A$-paths and Remark \ref{rk:Z-da-stacked-degrees}.
\end{proof}

We shall now consider part (2) of Condition~\ref{cond:comb-da}.

Let $T=\alpha _1\cdots \alpha _n$ be 
a closed $A$-trail in $\cQ $ with $\alpha _i\in \cQ _A$
for all $i$.  Assume that $n\pgq 2$ and
that $\rho _T=\set{\alpha _i\cdots \alpha _{i+d-1}\mid 1\ppq i\ppq
  n}\subseteq\rho $. Then $\ov T$ and all the paths lying on $\ov T$ are in $\B_{\lkd\Lambda }$ (none of their subpaths of length $d$ are in $\sigma $); those of length $\delta(k)$ for some $k\pgq0$ are in $\B_B$.

In a similar way to Section~\ref{subsec:cond-equiv-d-Koszul}, we make the following assumptions.
\begin{enumerate}[(i)]
\item none of the $\alpha _i$ are \aloops.
\item no \asubcycle of $T$ satisfies the same assumptions as $T$ (that is, there is no $A$-subcycle $q$ of $T$ of \alength at least $2$,  and $\rho _q\subseteq\rho $). 
\end{enumerate}

  \begin{prop}\label{cj-da-trails}
Let $T=\alpha _1\cdots \alpha _n$ be a closed $A$-trail with $n\pgq 2$, $\rho
_T\subseteq\rho $ and such  that assumptions~(i) and (ii)
hold.    Let $p$ be an \apath of \alength $d$ such that $\mg_\beta (\ov p)=0$ if
$\beta \in \cQ _A\setminus\set{\alpha _1,\ldots,\alpha _n}$ and which is not an
\asubpath of a power of $T$. Then $p\not\in\rho $.
  \end{prop}

  \begin{proof}
    The proof is very similar to that of Proposition \ref{cj-d-trails}, replacing paths with \apaths and using Consequence \ref{csq:atrail} in the proof that $d>k$.
  \end{proof}

  \begin{remark}\label{rk:d-trails} We keep the assumptions and notation of Proposition~\ref{cj-da-trails}.  Set $T_i=\alpha _i\cdots \alpha _n\alpha _1\cdots\alpha _{i-1}$. 
     Then for any $j\pgq 1$, we have
\begin{align*}
\ov T_i^j\ov\alpha _k\neq 0&\iff k=i-1\\
\ov \alpha _k\ov T_i^j\neq 0&\iff k=i.
\end{align*}
  \end{remark}

\begin{lem}\label{lm:d-a-stacked-power-closed-trail-Z}
 Let $T=\alpha _1\cdots\alpha _n$ be a closed $A$-trail that satisfies assumptions~(i) and (ii) and set $z_j=\sum_{i=1}^n\ov T_i^j$ with $nj=\delta (u)$ for some $u\pgq 1$. Then   $z_j\in Z$ if, and only if,  $T $ satisfies Condition~\ref{cond:comb-da}~(2).

Moreover, if $T $ does not satisfy  Condition~\ref{cond:comb-da}~(2), then no
element in $ B$ that is  homogeneous with respect to $\abs{\cdot}$ and $\mg$ (when viewed in $\lkd\Lambda $) and
that is a linear combination of non-trivial cycles lying on $\ov T$ is in $Z$.
\end{lem}

\begin{proof}
The proof is very similar to that of Lemma \ref{lm:dKoszul-power-closed-trail-Z}, replacing paths with \apaths, again using Remark \ref{rk:Z-da-stacked-degrees}, replacing the $d$-covering property by Consequence \ref{csq:subpaths-d-a-covering} and Proposition \ref{cj-d-trails} by Proposition \ref{cj-da-trails}. Note also that for the proof of the last part, testing commutation with paths in $B_2$ gives $s=n$ and $\lambda _k=\lambda _{k+d}$ for all $k$ and hence the result if $\abs z$ is odd; and if $\abs z$ is even, we must use the fact that $z$ commutes with elements in $B_3$ in a similar way to obtain, in addition, that $\lambda _k=\lambda _{k+d+1}$ for all $k$ and hence that $\lambda _k=\lambda _{k+1}$ for all $k$.
\end{proof}

\begin{proof}[Proof of Theorem \ref{tm:characterisations-fg-da-stacked}]
We note first that if $\gldim\Lambda$ is finite then $\Lambda$ satisfies \textbf{(Fg)} and Condition~\ref{cond:comb-da} is empty. The implication \ref{cond:J-da-stacked} $\Rightarrow$ \ref{cond:fg} is Theorem~\ref{tm:da-stacked-sufficient}. Again, the implication \ref{cond:Zgr} $\Rightarrow$ \ref{cond:Zgr-light} is clear and if, in addition, $K$ is algebraically closed, then the implication \ref{cond:fg} $\Rightarrow$ \ref{cond:Zgr} follows from \cite{ES}. It remains to prove that \ref{cond:Zgr-light} implies \ref{cond:J-da-stacked}. The proof is similar to that of Theorem~\ref{tm:characterisations-fg-d-koszul}, again replacing paths with \apaths (we need not assume that the integers $k$ are even). 
\end{proof}

\begin{remark} Suppose that $K$ is algebraically closed. We have now extended the
  equivalence between \ref{cond:fg} and \ref{cond:Zgr}, already known for Koszul
  algebras from \cite{ES}, as well as $d$-Koszul monomial algebras by Theorem
  \ref{tm:characterisations-fg-d-koszul}, to  $(D,A)$-stacked monomial algebras
  with $D\neq 2A$ whenever $A>1$ .

In particular, we can extend Corollary \ref{cor:Zgr-light} to $(D,A)$-stacked monomial algebras.
\end{remark}

\begin{corollary}\label{cor:DA-Zgr-light}
  Let $\Lambda $ be a $(D,A)$-stacked monomial algebra over an algebraically closed field with $D\neq 2A$ whenever $A>1$. Assume that
  $E(\Lambda )$ is a finitely generated $Z_{\text{gr}}(E(\Lambda
  ))$-module. Then the algebra $Z_{\text{gr}}(E(\Lambda ))$ is Noetherian.
\end{corollary}

\addappendix{The Ext algebra of a $(D,A)$-stacked monomial algebra}
\label{sec:ext-da-stacked-monomial}

{Leader and Snashall gave in \cite{LS} a presentation of the Yoneda
  algebra $E(\Lambda )$ of a $(D,A)$-stacked monomial algebra by quiver and
  relations. However, in our proof of Theorem
  \ref{tm:characterisations-fg-d-koszul} that \ref{cond:Zgr-light} implies
  \ref{cond:J-da-stacked} 
  for $d$-Koszul monomial algebras, we used the  description from  \cite[Sections 8 and 9]{GMMVZ} of  $E(\Lambda )$ as an algebra contained, as a graded vector space, in the Koszul dual $\tensor*[^{!}]{\Lambda}{} $. }
In this appendix, we generalise this description to $(D,A)$-stacked monomial algebras. 

Throughout this section, $\Lambda =K\cQ/I$ is a $(D,A)$-stacked   monomial
algebra with $D=dA$ and $d\pgq2$, 
where $I $ an
ideal generated by a set $\rho $ of \apaths of \alength $d=\frac{D}{A}$. We view $\Lambda $
as a
quotient of the tensor algebra: $\Lambda =\tns_{\Lambda _0}(\Lambda _1)/I$. 

All tensor products are taken over $\Lambda _0$ and we write  $\ot$ for $\otimes_{\Lambda _0}$. The subspace $\rln =I\cap(\Lambda
_1^{\ot D})=\spn(\rho )$ of $\tns =\tns _{\Lambda _0}(\Lambda _1)$ is a $\Lambda
_0$-$\Lambda _0$-submodule of $\Lambda _1^{\ot D}$; it is finite dimensional
over $K$. For an element $x\in \tns $, write $\ov x$ for its image in $\Lambda
$. Note that for $0\ppq i<D$ we have $\Lambda _i=\Lambda _1^{\ot i}$.

\subsection{Generalised Koszul complex of $\rln$}

Define  spaces $H_{\delta  (n)}\subseteq \tns _{\Lambda _0}^n(\Lambda _1)$  as follows: 
\[ H_0=\Lambda _0,\ H_1=\Lambda _1\text{ and, for }n\pgq 2,\ H_{\delta  (n)}=\bigcap_{i+j=\delta  (n)-d}(\Lambda _A^{\ot i })\ot \rln \ot (\Lambda _A^{\ot j}). \]
For $n\pgq 0$, let $\prj ^n$ be the right $\Lambda $-module defined by $\prj ^n=H_{\delta (n)}\ot \Lambda $; it is projective.

We have $H_{\delta (1)}=\Lambda _1=H_{\delta  (0)}\ot \Lambda _1$, $H_{\delta (2)}=\rln \subseteq \Lambda _1^{\ot D}=H_{\delta  (1)}\ot \Lambda _1^{\ot{D-1}}$ and, for any $n\pgq 3$,  $H_{\delta  (n)}\subseteq H_{\delta  (n-1)}\ot \Lambda _A^{\ot(\delta  (n)-\delta (n-1))}$. Indeed, for any $k\pgq 1$,
\begin{align*}
H_{\delta  (2k+2)}&=H_{(k+1)d}=\bigcap_{j=0}^{kd}(\Lambda _A^{\ot {(kd-j)}})\ot \rln \ot (\Lambda _A^{\ot j})\subseteq \bigcap_{j=d-1}^{kd}(\Lambda _A^{\ot (kd-j)})\ot \rln \ot (\Lambda _A^{\ot j})\\&=\bigcap_{j=0}^{(k-1)d+1}(\Lambda _A^{\ot ((k-1)d+1-j)})\ot \rln \ot (\Lambda _A^{\ot j})\ot (\Lambda _A^{\ot (d-1)})=H_{\delta  (2k+1)}\ot (\Lambda _A^{\ot (d-1)})\\
H_{\delta (2k+1)}&=H_{kd+1}=\bigcap_{j=0}^{kd+1}(\Lambda _A^{\ot {((k-1)d+1-j)}})\ot \rln \ot (\Lambda _A^{\ot j})\subseteq \bigcap_{j=A}^{kd+1}(\Lambda _A^{\ot ((k-1)d+1-j)})\ot \rln \ot (\Lambda _A^{\ot j})\\&=\bigcap_{j=0}^{kd}(\Lambda _A^{\ot ((k-1)d-j)})\ot \rln \ot (\Lambda _A^{\ot j})\ot \Lambda _A=H_{\delta  (2k)}\ot \Lambda _A
\end{align*} \sloppy It follows that the maps $F^1\colon \Lambda _1\ot \Lambda \rightarrow \Lambda _0\ot\Lambda \cong\Lambda $, $F^2\colon \Lambda _A^{\ot d}\ot\Lambda \rightarrow \Lambda _1\ot\Lambda $ and, for $n\pgq 3$,  $F^n\colon \Lambda _A^{\ot \delta  (n)}\ot \Lambda \rightarrow \Lambda _A^{\ot{\delta  (n-1)}}\ot \Lambda $ defined by 
\begin{align*}
F^1&(x_1\ot \lambda )=x_1\lambda \\
F^2&(x_1\ots x_D\ot\lambda )=x_1\ot x_2\cdots x_D\lambda \\
F^n&(y_1\ots y_{\delta (n)}\ot \lambda )=y_1\ots y_{\delta (n-1)}\ot y_{\delta (n-1)+1}\cdots y_{\delta (n)}\lambda ,
\end{align*}
where $x_i\in\Lambda _1$ and $y_i\in\Lambda _A$  for all $i$, induce maps $\diff^n\colon
\prj ^n\rightarrow \prj ^{n-1}$. More specifically,  for all $k\pgq1$,
\begin{align*}
F^{2k+1}(y_1\ots y_{kd+1}\ot \lambda )&= y_1\ots y_{kd}\ot y_{kd+1} \lambda\text{ if $n=2k+1$ is odd}\\
F^{2k+2}(y_1\ots y_{(k+1)d}\ot \lambda )&= y_1\ots y_{kd+1}\ot y_{kd+2}\cdots y_{(k+1)d} \lambda\text{ if   $n=2k+2$ is even.}
\end{align*}

Define also $\diff^0\colon \prj ^0=\Lambda _0\ot\Lambda \cong\Lambda \rightarrow \Lambda_0
$, which identifies with the natural
projection.

\sloppy Moreover, for $n\pgq 3$ we have $H_{\delta  (n+1)}\subseteq H_{\delta  (n)}\ot \Lambda _A^{\ot(\delta  (n+1)-\delta  (n))}\subseteq H_{\delta  (n-1)}\ot \Lambda _A^{\ot(\delta  (n)-\delta  (n-1))}\ot \Lambda _A^{\ot(\delta  (n+1)-\delta (n))}=H_{\delta (n-1)}\ot \Lambda _A^{\ot d}$ and $H_{\delta  (n+1)}\subseteq \Lambda _A^{\ot \delta  (n-1)}\ot \rln $, we have $H_{\delta (n+1)}\subseteq \left(\Lambda _A^{\ot \delta  (n-1)}\ot \rln \right)\cap \left(H_{\delta  (n-1)}\ot \Lambda _A^{\ot d}\right)=H_{\delta  (n-1)}\ot \rln $ (all the spaces involved are finitely generated and projective over $\Lambda _0$)  hence $\diff^n\circ \diff^{n+1}=0$. It is easy to check that $\diff^n\circ \diff^{n+1}=0$ when $n=1,$ $2$ or $3$.

Therefore we have a complex $(\prj ^n,\diff^n)$ of projective right $\Lambda $-modules.

\begin{thm}
  Let $\Lambda =K\cQ /I$ be a monomial algebra with $I$ generated in degree
  $D=dA$ with $d\pgq2$. Then $\Lambda $ is $(D,A)$-stacked monomial if, and only if,
  $(\prj ^\bullet,\diff^\bullet)$ is a minimal projective right $\Lambda $-module
  resolution of $\Lambda _0$.
\end{thm}

\begin{proof}
  By construction, $\prj ^n$ is generated in degree $\delta (n).$ Therefore, if
  $(\prj ^\bullet,\diff^\bullet)$ is a minimal projective right $\Lambda $-module
  resolution of $\Lambda _0$, then $\Lambda $ is a $(D,A)$-stacked monomial
  algebra.

Conversely, assume that  $\Lambda $ is a $(D,A)$-stacked monomial
  algebra. We already have a complex $(\prj ^\bullet,\diff^\bullet)$ of right $\Lambda
  $-modules such that $\prj ^n$ is generated in degree $\delta (n)$. The beginning $\prj ^1\xrightarrow {\diff^1}\prj ^0\xrightarrow {\diff^0}\Lambda
  _0\rightarrow 0$ is exact. 

We prove exactness at $\prj ^{2n+1}$ for $n\pgq 1$ (the proof of exactness at
$\prj ^{2n}$ and at $\prj ^1$ is similar, without the need for Consequence \ref{csq:subpaths-d-a-covering}). 

First note that $\diff^{2n+2}(\prj ^{2n+2})$ is generated in degree $\delta (2n+2)=(n+1)D$
in $\prj ^{2n+1}$. 

\sloppy Let $z=\sum_{i}x_{nd+1,i}\ot\cdots\ot x_{1,i}\ot \lambda _i$ be an element in
$\ker \diff^{2n+1}$ with $x_{j,i}\in \Lambda _A$ and $\lambda _i\in \Lambda $ for all
$i,j$. Then $\sum_i x_{nd+1,i}\ot\cdots\ot 
x_{2,i}\ot x_{1,i}\lambda _i$ is in $\tns \ot I$. It follows that $\lambda
_i\in\bigoplus_{k\pgq D-A}\Lambda _k$ and that $\ker \diff^{2n+1}$ is generated in
degrees at least $(n+1)D$. Therefore $z$ can be rewritten as
$z=\sum_{i}x_{nd+1,i}\ots x_{1,i}\ot y_{d-1,i}\cdots y_{1,i}\lambda _i'$ with
$y_{j,i}\in\Lambda _A$ and $\lambda _i'\in \Lambda $ for all $i,j$, with the
$y_{j,i}$ right uniform and $\mt(y_{1,i})\neq\mt(y_{1,k})$ when $i\neq
k$. Write $\lambda _i'=\sum_{l\pgq 0}\lambda _{i,l}'$ with $\lambda _{i,l}'\in
\Lambda _l$ for all $i,l$, and $\lambda _{i,0}'=\mt(y_{1,i})$. Then each of
the $\sum_{i}x_{nd+1,i}\ots x_{1,i}\ot y_{d-1,i}\cdots y_{1,i}\lambda _{i,l}'$
is in $\ker \diff^{2n+1}$ so in particular $z':=\sum_{i}x_{nd+1,i}\ots x_{1,i}\ot y
_{d-1,i}\cdots y_{1,i}\in \ker \diff^{2n+1}$.

Consider $z'':=\sum_{i}x_{nd+1,i}\ots x_{1,i}\ot y_{d-1,i}\ots y_{1,i}\ot
\mt(y_{1,i})\in \Lambda _A^{\ot((n+1)d)}\ot \Lambda  $. We show that $z''\in
\prj ^{2n+2}$; this will imply that $z'=\diff^{2n+2}(z'')\in \im \diff^{2n+2}$. Since
$\Lambda $ is monomial and $\diff^{2n+1}(z')=0$, \sloppy we may assume that that for all $i$,   $
x_{1,i}y_{d-1,i}\cdots y_{1,i}$ is a path; since it is in $I$ and has degree $D$,  it is in
$\rho =\R^2$. By definition of $\prj ^{2n+1}$, we may assume that $z$ is written so that for each $i$,  $x_{d,i}\cdots x_{1,i}$ is a path in $\rho =\R^2$. The path    $
x_{1,i}y_{d-1,i}\cdots y_{1,i}$  properly overlaps $x_{d,i}\cdots x_{1,i}$ therefore, using
Consequence \ref{csq:subpaths-d-a-covering}, it follows that for all $k$ with $1\ppq
k\ppq d-1$ we have $x_{k,i}\cdots x_{1,i}y_{d-1,i}\cdots y_{k,i}\in \rho
$ and hence $z''\in \bigcap_{k=1}^{d-1}(\Lambda _A^{\ot (nd-k+1)})\ot \rln \ot (\Lambda _A^{\ot (k-1)})\ot \Lambda $. Finally, using the fact that $z\in \prj ^{2n+1}=H_{nd+1}\ot\Lambda $, we get $z''\in
H_{(n+1)d}\ot \Lambda =\prj ^{2n+2}$. We have proved that $(\ker
\diff^{2n+1})_{(n+1)d}\subseteq \im \diff^{2n+2}$.

Since $\im \diff^{2n+2}$ is generated in degree  $(n+1)d$ and $\ker \diff^{2n+1}$ is
generated in degrees at least $(n+1)d$, it follows that $\ker \diff^{2n+1}\subseteq \im
\diff^{2n+2}$ and finally that $\ker \diff^{2n+1}=\im{\diff^{2n+2}}$.

Finally, since $\im \diff^{n+1}$ is generated in degree $\delta (n+1)$ and $\prj ^n$ is
generated in degree $\delta (n)<\delta (n+1)$,  $\im \diff^{n+1}\subseteq \mathfrak r
\prj ^n$ for all $n$ and therefore the resolution is minimal. 
\qedhere

\end{proof}

\subsection{The Ext algebra }

\subsubsection{Some duality results}\label{sec:duality}

We recall without proof some results stated in \cite{bgs}. All modules are finitely generated $\Lambda _0$-$\Lambda _0$-bimodules. All claims are easily checked for
bimodules that are free and finitely generated as left or as right $\Lambda
_0$-modules, and follow for arbitrary finitely generated modules (since $\Lambda _0$ is semisimple,
all the $\Lambda _0$-modules (left or right) are projective).

For any bimodules $V$ and $W$, define $V^*=\Hom_{\Lambda _0-}(V,\Lambda _0)$ and
$\ld W=\Hom_{-\Lambda _0}(W,\Lambda _0)$; they are $\Lambda _0$-$\Lambda
_0$-bimodules, for the actions given for all $e,e'$ in $\Lambda _0$, 
$\alpha \in V^*$, $\beta \in\ld W$ and  $v\in
V$, $w\in W$ by:
\[(e\alpha e')(v)=\alpha (ve)e'\text{ and }(e\beta e')(w)=e\beta (e'w).\]

 There are natural isomorphisms of $\Lambda _0$-$\Lambda
_0$-bimodules $\Lambda _0^*\cong \Lambda _0\cong \ld \Lambda _0$ which we view
as identifications.

There are also natural bimodule isomorphisms $V\cong \ld{(V^*)}$ and $W\cong (\ld W)^*$ of
bimodules.

\sloppy If $V_1\subseteq V$ (respectively $W_1\subseteq W$), define
$V_1^\perp=\set{\alpha \in V^*\mid \alpha (V_1)=0}$
 (respectively $\lo W_1=\set{\beta \in \ld W\mid \beta  (W_1)=0}$). If $V_1$
 (respectively $W_1$) is a sub-bimodule, then they are
 sub-bimodules of $V^*$ and $\ld W$ respectively. 

Fix a $\Lambda _0$-$\Lambda _0$-bimodule $V$. Then:
\begin{enumerate}[(i)]
\item\label{item:dual-tensor} if $W$ is another $\Lambda _0$-$\Lambda _0$-bimodule, then there is an
  isomorphism of $\Lambda _0$-$\Lambda_0$-bimodules $\varphi _{V,W}\colon \ld
  V\ot \ld W\rightarrow \ld{(W\ot V)}$ given for all $\alpha \in \ld V$, $\beta
  \in\ld W$, $v\in V$ and $w\in W$ by $\varphi _{V,W}(\alpha \ot \beta )(w\ot
  v)=\alpha (\beta (w)v)$. There is a similar isomorphism with right duals, which sends $\alpha \ot\beta $ to the map $w\ot v\mapsto \beta (w\alpha (v))$;
\item\label{item:orthogonal-sum} if $U$ and $W$ are sub-bimodules of $V$, then $\lo{(U+W)}=\lo U\cap \lo W$
  and $(U+W)^\perp=U^\perp\cap W^\perp$;
\item\label{item:dual-quotient}  if $U$ is a sub-bimodule of $V$, then $(V/U)^*\cong U^\perp$ and
  $\ld(V/U)\cong \lo U$; 
\item\label{item:dual-idempotents} if $W$ is a sub-bimodule of $V$, then for any idempotents $e_i$, $e_j$
  with $(i,j)\in \cQ _0^2$, we have $e_i \ld W e_j=\ld(e_jWe_i)$;
\item\label{item:dual-dimension}  if $U$ is a sub-bimodule of $V$, then for all $i,j$ in $\cQ _0$ we have $\dim e_j\lo Ue_i=\dim e_iVe_j-\dim
  e_iUe_j=\dim e_jU^\perp e_i$;
\item\label{item:dual-double-orthogonal} if $U$ is a sub-bimodule of $V$, then   under the identification of $V$
  with $\ld{( V^*)}$, we have $\lo(U^\perp)$ and   under the identification of $V$
  with $(\ld V)^*$, we have $(\lo U)^\perp$;
\item \label{item:orthogonal-tensor}  if $U$ is a sub-bimodule of $V$ and $W$ and $Z$ are $\Lambda _0$-$\Lambda
  _0$-bimodules, there are bimodule isomorphisms $\lo(W\ot U\ot Z)\cong\ld Z\ot
  \lo U\ot \ld W$ and
  $(W\ot  U\ot Z)^\perp\cong Z^*\ot U^\perp\ot W^*$.
\end{enumerate}

\subsubsection{Description of the Ext algebra}

From the above, there is a natural isomorphism $\ld(\Lambda _A^{\ot i})\cong
(\ld\Lambda _A)^{\ot i}$, which we view as an identification. We also view $\rln $ as a subspace of $\Lambda _A^{\ot d}$ (rather than $\Lambda _1^{\ot D}$). We may then
consider $\lo \rln =\set{f\in (\ld\Lambda _A)^{\ot d}\mid f(x)=0\text{ for all
  }x\in \rln }$. The dual algebra of $\Lambda $ is then $\lkd\Lambda =\tns _{\Lambda
  _0}(\ld \Lambda _A)/(\lo \rln )$. It is a graded $d$-homogeneous algebra since
$\lo \rln  $ is contained in $(\ld\Lambda _A)^{\ot d}$, therefore $\lkd\Lambda
=\bigoplus_{n\pgq 0}\lkd\Lambda _n$.

In terms of quivers, we have $\Lambda _0=K \cQ _0$ and $\Lambda _A=K \cQ _A$, the vector space whose basis is the set $\cQ _A$ of paths of length $A$ in $\cQ $;
moreover, $\ld \Lambda _A\cong K \cQ _A^{\text{op}}$ using
\ref{item:dual-idempotents}. 

Then $\lkd \Lambda $ is isomorphic to $K \Gamma /(\lo\rho  )$ where $\Gamma $ is
the quiver with the same vertices as $\cQ $ and whose set of  arrows is $\Gamma
_1=\set{\ov\alpha \colon i\rightarrow j\mid \text{there exists } \alpha \in \cQ _A, \
  \alpha \colon j\rightarrow i}$ and where $\lo\rho  $ is the left orthogonal of
the set $\rho $ viewed as a set of $A$-paths,  for the bilinear form $K\Gamma
_d\times K(\cQ _A)_d\rightarrow K $ defined on paths of length $d$ in $\Gamma $ and \apaths of \alength $d $ in $\cQ $ by $\rep{\ov\alpha _d\cdots\ov\alpha _1,\beta _1\cdots \beta _d}=1$ if $\alpha _1\cdots\alpha _d=\beta _1\cdots\beta _d$ and $0$ otherwise, where the $\alpha _i$ and $\beta _i$ are paths of length $A$.  

The algebra $K\Gamma /{(\lo\rho )}$ is monomial,  and the ideal $(\lo\rho )$ has
a minimal generating set $\sigma $ given by all the paths $\ov \alpha
_d\cdots\ov\alpha _1$ such that the \apath $\alpha _1\cdots\alpha _d$ is not in
$\rho $. In particular,  if $\ov\gamma =\ov\alpha _r\ldots \ov\alpha _1$ is a
path in $\Gamma $, with $\alpha _i\in \cQ _A$ for all $i$, then $\ov\gamma \not\in (\sigma )$ if, and only if, for each $i$ with $1\ppq i\ppq r-d+1$, the path ${\alpha _i}\cdots {\alpha _{i+d-1}}$ is in $\rho $ (we use the fact that $\Lambda $ is monomial here). 

\begin{lem}
  \label{lm:dual-Bn}
There is an isomorphism $(\lkd\Lambda _{\delta (n)})^*\cong H_{\delta (n)}$.
\end{lem}

\begin{proof}
  By definition, $\lkd\Lambda _{\delta (n)}=\dfrac{(\ld\Lambda _A)^{\ot \delta 
      (n)}}{\sum_{i=0}^{\delta (n)-d}(\ld\Lambda _A)^{\ot {(\delta (n)-d-i)}}\ot
    \lo \rln \ot (\ld\Lambda _A)^{\ot i}}$. Therefore  we have 
\begin{align*}
 (\lkd\Lambda _{\delta (n)})^*&\cong {\left(\sum_{i=0}^{\delta (n)-d}(\ld\Lambda
                              _A)^{\ot (\delta (n)-d-i)}\ot    \lo \rln \ot (\ld\Lambda _A)^{\ot i}\right)}^\perp\\
 &=\bigcap_{i=0}^{\delta  (n)-d}{\left((\ld\Lambda _A)^{\ot (\delta (n)-d-i)}\ot     \lo \rln \ot (\ld\Lambda _A)^{\ot i}\right)}^\perp\\
 &=\bigcap_{i=0}^{\delta (n)-d}\left({((\ld\Lambda _A)^{\ot i})}^*\ot {(\lo \rln )}^\perp\ot{((\ld\Lambda _A)^{\ot  (\delta (n)-d-i)})}^*\right)\\
&\cong\bigcap_{i=0}^{\delta (n)-d}\left(\Lambda _A^{\ot i}\ot \rln \ot\Lambda _A^{\ot
  (\delta (n)-d-i)}\right)\\
&=H_{\delta (n)}.
\end{align*} This isomorphism takes $x=\sum x_1\ots  x_{\delta (n)}\in H_{\delta (n)}$ to the map
$g_x\colon \lkd\Lambda _{\delta (n)}\rightarrow \Lambda _0$ defined by
\[  g_x(\ov{\gamma _{\delta (n)}\ots \gamma _{1}})=\sum \gamma _{\delta (n)}(\gamma _{\delta (n)-1}(\ldots (\gamma _2(\gamma _1(x_1)x_2)\ldots)x_{\delta (n)-1})x_{\delta (n)})\]
 where the $x_i$ are in $\Lambda _A$ and
the $\gamma _i$ are in $\ld\Lambda _A$.
\end{proof}

\begin{lem}
\label{lm:isomorphisms}
  There is an isomorphism $\psi \colon \lkd\Lambda _{\delta (n)}\rightarrow
  \Hom_\Lambda (\prj ^n,\Lambda _0)$ given by \[\psi (\ov f)(x_1\ots  x_{\delta 
    (n)}\ot \lambda )=f_{\delta (n)}(f_{\delta (n-1)}(\ldots (f_1(x_1))\ldots
  x_{\delta (n)-1})x_{\delta (n)})\ov\lambda \]
  where $\ov f=\ov{f_{\delta (n)}\ots f_{1}}
  \in \lkd\Lambda _{\delta (n)}$ with
  $f_i\in \ld\Lambda _1$ for all $i$, $x_i\in \Lambda _1$ for all $i$ and
  $\lambda \in \Lambda $.
\end{lem}

\begin{proof}
  The isomorphism $\psi $ is the composition of the following isomorphisms: 
\begin{itemize}
\item $\lkd\Lambda _{\delta (n)}\rightarrow \Hom_{-\Lambda _0}((\lkd\Lambda _{\delta 
  (n)})^*,\Lambda _0)$, which sends $ \ov f$ to the map $ [g\mapsto g(\ov f)]$;
\item $\Hom_{-\Lambda _0}((\lkd\Lambda _{\delta 
  (n)})^*,\Lambda _0)\rightarrow \Hom_{-\Lambda _0}(H_{\delta (n)},\Lambda _0)$,
which sends a map $h$ to the map $[x\mapsto h(g_x)]$, where $g_x$ is as in the
proof of Lemma \ref{lm:dual-Bn};
\item $\Hom_{-\Lambda _0}(H_{\delta (n)},\Lambda _0)\rightarrow
  \Hom_{-\Lambda}(H_{\delta (n)}\ot \Lambda ,\Lambda _0)$, which sends a map $k$
  to the map $[x\ot \lambda \mapsto k(x)\ov\lambda ]$.
\end{itemize}
Applying these isomorphisms to $\ov f$ gives the expression in the
statement. 
\end{proof}

Let $B$ be the vector space $B=\bigoplus_{n\pgq 0}B_n$ where $B_n=\lkd\Lambda
_{\delta (n)}$. Define a multiplication on $B$ as follows: for $x\in B_n$ and $y\in B_m$, set 
\[ x.y=
\begin{cases}
0&\text{ if $n$ and $m$ are odd};\\
0&\text{ if $n$ or $m$ is equal to $1$ and $n\pgq 1$, $m\pgq 1$};\\
xy&\text{ in $\lkd\Lambda $ otherwise}. 
\end{cases}
 \] 
 The algebra $B$ is a graded $K$-algebra generated in degrees $0$, $1$, $2$ and $3$.

We want to prove that   $E(\Lambda )\cong B$ {when $A>1$ and $D\neq 2A$}. We first need a description of the
Yoneda product.

\begin{prop}
\label{prop:cup-product}
 Let $\Lambda $ be a $(D,A)$-stacked monomial algebra with $A>1$, $D=dA$ and $d\pgq 3$.
    The Yoneda product of $f_n\in\Hom_{\Lambda }(\prj ^n,\Lambda _0)$ and $f_m\in\Hom_{\Lambda
  }(\prj ^m,\Lambda _0)$ is given by 
\[ f_nf_m=
\begin{cases}
0\text{ if $n$ and $m$ are odd}\\
0\text{ if $n\pgq1$, $m\pgq 1$ and $n=1$ or $m=1$}\\
\sum x_1\ots x_{\delta (n)+\delta (m)}\ot\lambda \mapsto f_n(\sum f_m(x_1\ots
x_{\delta (m)}\ot 1)x_{\delta (m)+1}\\
\hspace*{2.5in}\ot x_{\delta
  (m)+2}\ots x_{\delta (m)+\delta (n)}\ot \lambda )\text{ otherwise,}
\end{cases}
 \] where the $x_i$ are all in $\Lambda _A$.

\end{prop}

\begin{proof}
  If $m$ and $n$ are odd or if $m\pgq 1$
  or $n\pgq 1$ is equal to $1$ then, {under the assumption that $A>1$ and  $D\neq 2A$,} the Yoneda products vanish by \cite[Theorem
  3.4]{LS}. We now assume that $m$ or $n$ is even and that $m\neq 1$ and
  $n\neq 1$. 

Let $\sigma \colon \Lambda _0\rightarrow \Lambda $ be the natural inclusion. 

Consider $f_m\colon \prj ^m\rightarrow \Lambda _0$; it lifts to $f_m^0=\sigma \circ
f_m\colon \prj ^m\rightarrow \Lambda $. We  now define further liftings
$f_m^i\colon  \prj ^{m+i}\rightarrow \prj ^i$ for $i\pgq 1$ as follows: 
\begin{align*}
f_m^1(x_1\ots x_{\delta (m+1)}\ot\lambda )&=f_m^0(x_1\ots x_{\delta (m)}\ot 1 )x_{\delta (m)+1}\ot  x_{\delta (m)+2}\cdots x_{\delta (m+1)}\lambda \\
f_m^i(y_1\ots y_{\delta (m+i)}\ot \lambda )&  =f_m^0(y_1\ots y_{\delta (m)}\ot 1)y_{\delta (m)+1}\ots    y_{\delta (m+i)} \ot \lambda \\&\hphantom{\qquad\qquad\qquad\qquad\ot y_{\delta  (m+i-1)+2}\cdots y_{\delta (m+i)}\lambda \qquad}\text{if $m$ or $i\pgq 2$ is even}\\
f_m^i(y_1\ots y_{\delta (m+i)}\ot \lambda )&  =f_m^0(y_1\ots y_{\delta (m)}\ot  1)y_{\delta (m)+1}\ots   y_{\delta (m+i-1)+1}\\
&\qquad\qquad\qquad\qquad\ot y_{\delta  (m+i-1)+2}\cdots y_{\delta (m+i)}\lambda \qquad\text{if $m$ and $i\pgq2$ are odd}
\end{align*} where the $x_i$ are in $\Lambda _1$  and the $y_i$ are in $\Lambda _A$. The proof that  $(f_m^i)_{i\pgq 0}$ is a family of liftings of
$f_m$, that is, $f^{i-1}\circ \diff^{i+m}=\diff^i\circ f_m^i$ for all $i\pgq 1$, is tedious but straightforward.

Finally, if $n$ or $m$ is even and $n\pgq2$, $m\pgq 2$ ,then
\begin{align*}
f_nf_m(y_1&\ots y_{\delta (m)+\delta (n)}\ot \lambda )=f_n\circ f_m^n(y_1\ots y_{\delta (m+n)}\ot \lambda )\\
&=f_n(f_m^0(y_1\ots y_{\delta (m)}\ot 1)y_{\delta (m)+1}\ot y_{\delta (m)+2}\ots y_{\delta  (m)+\delta (n)}\ot \lambda )\\
&=f_n(f_m(y_1\ots y_{\delta (m)}\ot  1)y_{\delta (m)+1}\ot y_{\delta  (m)+2}\ots y_{\delta   (m)+\delta (n)}\ot \lambda ).
\qedhere\end{align*}

\end{proof}

\begin{thm}
If $\Lambda $ is a $(D,A)$-stacked monomial algebra, with { $A>1$, $D=dA$ and $d\pgq3$} 
, then $E(\Lambda )$ is isomorphic to $B$ as graded algebras. In particular,
$\Ext_\Lambda ^n(\Lambda _0,\Lambda _0)$ is isomorphic to $\lkd \Lambda _{\delta (n)}$.
\end{thm}

\begin{proof}
We use the
  isomorphisms in Lemma \ref{lm:isomorphisms} and the cup-product described in Proposition \ref{prop:cup-product}. If $\ov f=\ov {f_{\delta (n)}\ots
    f_{1}}\in B_n$ and $\ov g=\ov{g_{\delta (m)}\ots g_{1}}\in
  B_m$, where $m$ or $n$ is even and both are at least $2$, then 
  \begin{align*}
     \psi (\ov f)&\psi (\ov g)(y_1\ots y_{\delta (m)+\delta (n)}\ot \lambda )=\psi   (\ov f)\left(\psi  (\ov g)(y_1\ots y_{\delta  (m)}\ot 1)y_{\delta (m)+1}\ots y_{\delta  (m+n)}\ot \lambda \right)\\
 &=f_{\delta (n)}(f_{\delta (n)-1}(\ldots (f_2(f_1(\psi (\ov g)(y_1\ots
   y_{\delta (m)}\ot1)y_{\delta (m)+1})y_{\delta (m)+2})\ldots)y_{\delta
 '  (m+n)-1})y_{\delta (m+n)})\ov \lambda \\
&=f_{\delta (n)}(f_{\delta (n)-1}(\ldots (f_2(f_1(g_{\delta 
  (m)}(\ldots(g_1(y_1))\ldots y_{\delta (m)})  y_{\delta (m)+1})y_{\delta (m)+2})\ldots)y_{\delta 
   (m+n)-1})y_{\delta (m+n)})\ov \lambda\\
&=\psi (\ov {fg})(y_1\ots y_{\delta (m)+\delta (n)}\ot \lambda )\\
&=\psi (\ov f\cdot\ov g)(y_1\ots y_{\delta (m)+\delta (n)}\ot \lambda )
  \end{align*} therefore $\varphi (\ov f\cdot \ov g)=\psi (\ov f)\psi (\ov g)$ and we have proved that $\psi $ is an isomorphism of graded algebras.
\end{proof}

\begin{remark}
  \label{rk:basis-ext-algebra}
Recall from Subsection~\ref{subsec:Ext algebra} that the $m$-th projective in a minimal projective right $\Lambda $-module resolution of $\Lambda _0$ is $L^m=\bigoplus_{R_i^m\in\R^m}\mt(R_i^m)\Lambda $. By Consequence \ref{csq:da-covering}, $R_i^m\in \prj ^m$. We then have an isomorphism $\prj ^m\rightarrow L^m$ which is determined by $R_i^m\mapsto \mt(R_i^m)$ for all $i$. 

As we mentioned in Subsection~\ref{subsec:Ext algebra}, the authors of \cite{GZ} 
 also gave a basis of $E(\Lambda )$, namely the set $\set{g_i^m\in\Hom_\Lambda (L^m,\Lambda _0)\mid R_i^m\in\R^m}$ where $g_i^m(\mt(R_j^m))=\mt(R_i^m)$ if $j=i$ and $0$ otherwise. The element $g_i^m$ corresponds to a map in $\Hom_\Lambda (\prj ^m,\Lambda _0)$ which we denote again by $g_i^m$ and that is defined by $g_i^m(R_j^m)=\mt(R_i^m)$ if $j=i$ and $0$ otherwise. 

We have isomorphisms $K\cQ _A\cong K\Gamma _1$ and $K\cQ _A\cong \ld\Lambda _A=\ld{(K\cQ _A)}$ which combine to the isomorphism which associates to $\ov\alpha \in \Gamma _1$  the linear form $f_\alpha $ on $K\cQ _A$ that sends $\beta \in \cQ _A$ to $\mt(\alpha )$ if $\beta =\alpha $ and to $0$ otherwise. This extends to an isomorphism between the algebras $K\Gamma /(\sigma )$ and $\lkd\Lambda $ that sends a path $\ov p$ of length $n$ to the class of the linear map $f_p\in (\ld\Lambda _A)^{\ot n}$ defined on \apaths by $f_p(q)=\mt(p)$ if $q=p$ and $0$ otherwise.

Now consider $R_i^m=\alpha _1\cdots\alpha _{\delta (m)}$ with $\alpha \in \cQ _A$ for all $i$ and $\ov R_i^m=\ov \alpha _{\delta (m)}\cdots \ov\alpha _1$ in $K\Gamma /(\sigma )$. It is easy to check that $g_i^m=\psi \left(\ov{f_{\alpha _{\delta (m)}}\ots f_{\alpha _1}}\right)$ so that it corresponds, via the isomorphism above, to $\ov R_i^m$.

Therefore we have a basis $\B_B=\set{\ov R_i^m\mid R_i^m\in\R^m}$ of $B$.
\end{remark}


\begin{thebibliography}{99}

\bibitem{B} \textsc{M.J. Bardzell}, The alternating syzygy behavior of monomial algebras,  \textit{J. Algebra} \textbf{188} (1997), p~69-89.



\bibitem{bgs} \textsc{A. Beilinson, V. Ginzburg} and \textsc{W. Soergel},
  Koszul duality patterns in representation theory, \textit{J. Amer. Math. Soc.} \textbf{9}
  (1996), p~473-527.
  
  \bibitem{berger} \textsc{R. Berger}, Koszulity for nonquadratic algebras,  \textit{J. Algebra} \textbf{239} (2001), p~705-734.



\bibitem{dgt} \textsc{V. Dotsenko, V. Gélinas} and \textsc{P. Tamaroff,} Finite generation for Hochschild cohomology of Gorenstein monomial algebras, \textit{Selecta Math.} \textbf{29} (2023), 45 pages.


  
\bibitem{Erdmann} \textsc{K. Erdmann}, Algebras with non-periodic bounded modules, 
\textit{J. Algebra} \textbf{475} (2017), p~308-326.

\bibitem{EHSST} \textsc{K. Erdmann}, \textsc{M. Holloway}, \textsc{N. Snashall},
  \textsc{\O. Solberg} and \textsc{R. Taillefer}, Support varieties for
  selfinjective algebras, \textit{K-Theory} \textbf{33} (2004), p~67-87.

\bibitem{ES} \textsc{K. Erdmann} and \textsc{Ø. Solberg}, Radical cube zero weakly symmetric algebras and support varieties, \textit{J. Pure Appl. Algebra} \textbf{215} (2011), p~185-200.

\bibitem{FS} \textsc{T. Furuya} and \textsc{N. Snashall}, Support varieties for modules over stacked monomial algebras, \textit{Comm. Alg.} {\bf 39} (2011), p~2926-2942.

\bibitem{GHZ} \textsc{E.L. Green}, \textsc{D. Happel} and \textsc{D. Zacharia}, Projective resolutions over Artin algebras with zero relations, \textit{Illinois J. Math.} \ {\bf 29} (1985), p~180-190.


\bibitem{gmv}  \textsc{E. L. Green} and \textsc{R. Martinez-Villa}, Koszul and Yoneda algebras II, \textit{CMS Conf Proc Algebras and modules, II (Geiranger, 1996) Amer. Math. Soc.} \textbf{24} (1998), p~227-244.


\bibitem{GMMVZ} \textsc{E. L. Green, E. N. Marcos, R. Martínez-Villa} and \textsc{P. Zhang}, D-Koszul algebras, \textit{J. Pure Appl. Algebra} \textbf{193} (2004), p~141-162.

\bibitem{GS-J Alg} \textsc{E.L. Green} and \textsc{N. Snashall}, Finite generation of Ext for a generalization of $D$-Koszul algebras, \textit{J. Algebra} {\bf 295} (2006), p~458-472.

\bibitem{GS-colloq math} \textsc{E. L. Green} and \textsc{N. Snashall}, The Hochschild
  cohomology ring modulo nilpotence of a stacked monomial algebra,
  \textit{Colloq. Math.} \textbf{105} (2006), p~233-258.


\bibitem{GZ} \textsc{E. L. Green} and \textsc{D. Zacharia}, The cohomology ring of a monomial algebra, \textit{Manus. Math.} \textbf{85} (1994), p~11-23.

\bibitem{J} \textsc{R. Jawad}, Cohomology and finiteness conditions for generalisations of Koszul algebras,
PhD thesis, University of Leicester (2019).

\bibitem{KPS} \textsc{J. Külshammer}, \textsc{C. Psaroudakis} and
  \textsc{\O. Skartsæterhagen}, Derived invariance of support varieties,  \textit{Proc. Amer. Math. Soc.} \textbf{147} (2019), p~1-14.

\bibitem{LS} \textsc{J. Leader} and \textsc{N. Snashall}, The Ext Algebra and a New Generalization of D-Koszul Algebras, \textit{Quart. J. Math.} \textbf{68} (2016), p~433-458. 

\bibitem{pp} \textsc{A. Polishchuk} and \textsc{L. Positselski}, Quadratic algebras, \textit{Univ. Lecture Ser.} \textbf{37}, \textit{American Mathematical Society, Providence} (2005).

\bibitem{PSS} \textsc{C. Psaroudakis}, \textsc{\O. Skartsæterhagen} and \textsc{\O. Solberg}, Gorenstein categories, singular equivalences and finite generation of cohomology rings in recollements, \textit{Trans. Amer. Math. Soc. Ser. B} \textbf{1} (2014), p~45-95.

\bibitem{SchrollS} \textsc{S. Schroll} and \textsc{N. Snashall}, Hochschild
  cohomology and support varieties for tame Hecke algebras, 
\textit{Quart. J. Math.} \textbf{62} (2011),  p~1017~1029.

\bibitem{Sharp} \textsc{R. Y. Sharp}, Steps in Commutative Algebra, \textit{London Math. Soc. Stud. Texts} {\bf 19}, \textit{Cambridge University Press, Cambridge} (1990).

\bibitem{Sk} \textsc{\O. Skartsaeterhagen}, Singular equivalence and the (Fg) condition, \textit{J. Algebra} \textbf{452} (2016), p~66-93.


\bibitem{SS} \textsc{N. Snashall} and \textsc{\O. Solberg}, Support varieties and Hochschild cohomology rings, \textit{Proc. London Math. Soc. (3)} \textbf{88} (2004),  p~705-732.

\bibitem{ST1} \textsc{N. Snashall} and \textsc{R. Taillefer}, Hochschild cohomology of socle deformations of a class of Koszul self-injective algebras, \textit{Colloq. Math.} \textbf{119} (2010),  p~79-93. 

\bibitem{ST2} \textsc{N. Snashall} and \textsc{R. Taillefer}, The Hochschild cohomology ring of a class of special biserial algebras, \textit{J. Algebra Appl.} \textbf{9} (2010),  p~73–122.



\bibitem{W} \textsc{S. Witherspoon}, Varieties for modules of finite dimensional
  Hopf algebras, \textit{Geometric and topological aspects of the representation
    theory of finite groups, Springer Proc. Math. Stat.} \textbf{242} (2018), p~481-495.

\end{thebibliography}
\end{document}